\documentclass{article}
\usepackage{amsmath}
\usepackage{paralist}
\usepackage[misc]{ifsym}
\usepackage{epsfig} 
\usepackage{comment}
\usepackage{amsthm}
\usepackage{breqn}
\usepackage{graphicx} 
\usepackage{physics}
\usepackage{xcolor}
\usepackage{amsfonts}
\usepackage[normalem]{ulem}
\usepackage{amsmath,amsfonts}
\usepackage{longtable}
\usepackage{geometry}
\usepackage{booktabs}
\usepackage{caption}
\usepackage[toc,page]{appendix}
\usepackage[colorlinks=true]{hyperref}
  
\usepackage[pagewise]{lineno}
\hypersetup{urlcolor=blue, citecolor=red}
\allowdisplaybreaks

\newcommand{\nortwo}[1]{\norm{#1}_{_{2}}}

\newcommand{\spt}[1]{\langle #1 \rangle_{_{2}}}

\newcommand{\spta}[1]{\left\langle #1 \right\rangle_{2}}
\newcommand{\sptb}[1]{\Bigg\langle #1 \Bigg\rangle_{2}}
\newcommand{\eequ}{\end{equation}}
\newcommand{\bequ}{\begin{equation}}
\newcommand{\eequd}{\end{eqnarray*}}
\newcommand{\bequd}{\begin{eqnarray*}}

\def\R{\mathbb{R}}
\def\M{\mathcal{M}}
\def\E{\mathrm{E}}
\newtheorem{theorem}{Theorem}[section]

\newtheorem{lemma}[theorem]{Lemma}

\theoremstyle{definition}

\newtheorem{remark}[theorem]{Remark}

\def\R{\mathbb{R}}
\title{Condition number for finite element discretisation of nonlocal PDE systems with applications to biology}

\begin{document}

\author{\centerline{\scshape
Olusegun E. Adebayo$^{{\href{mailto:olusegun.adebayo@univ-fcomte.fr}{\textrm{\Letter}}}\,1}$, 
 Raluca Eftimie$^{{\href{mailto:raluca.eftimie@univ-fcomte.fr; r.a.eftimie@dundee.ac.uk}{\textrm{\Letter}}}*1,2}$ and Dumitru Trucu$^{{\href{mailto:D.Trucu@dundee.ac.uk}{\textrm{\Letter}}}*2}$
 }
 }
\maketitle

\medskip

{\footnotesize
 \centerline{$^1$Laboratoire de Math\'{e}matiques de Besan\c{c}on, UMR CNRS 6623,}
} 
{\footnotesize\centerline{University of Franche-Comt\'{e}, Besan\c{c}on 25000, France}}
\medskip

{\footnotesize
 \centerline{$^2$Division of Mathematics, University of Dundee, Dundee, DD1 4HN, United Kingdom}
}

\begin{abstract} 
 In this work, we investigate the condition number for a system of coupled non-local reaction-diffusion-advection equations developed in the context of modelling normal and abnormal wound healing. 
 Following a finite element discretisation of the coupled non-local system, we establish bounds for this condition number. 
 
We further discuss how model parameter choices affect the conditioning of the system. Finally, we discuss how the step size of the chosen time-stepping scheme and the spatial grid size of the finite element methods affect the bound for the condition number, while also suggesting possible parameter ranges that could keep the model well conditioned.
\end{abstract}

\section{Introduction}
The condition number of the associated matrix to a partial differential equation (PDE) is crucial to understanding the stability and robustness of its solution to perturbation in its inputs~\cite{ErnGuermond2006}. A small change in the matrix of an ill-conditioned PDE will result in large changes in the approximation of the solution, while for a well-conditioned PDE a small change in the associated matrix would lead to a small change in the approximated solution~\cite{ErnGuermond2006}. A known fact is that the finite element method for solving PDEs generally leads to linear systems involving large and sparse matrices that are usually solved using iterative techniques, and the condition number of the associated matrix can be used to estimate the rate of convergence of such iterative techniques~\cite{ErnGuermond2006, Bathe2014}. 
Furthermore, while solving linear systems of equations (assuming a direct solver is used) the number of significant digits in the solution decreases as the condition number increases (i.e., the number of digits that one cannot rely upon is proportional with $\log(condition\textrm{ }number)$) \cite{Bathe2014,dalhquist74}. It is also important to note that when dealing with iterative solution schemes, the number of iterations required to attain a pre-established error threshold also depends on the condition number, as systems with high condition numbers require more iterations \cite{Bathe2014}.

We emphasize that a variety of numerical studies have investigated the condition numbers of PDEs approximated using finite elements; see for example,~\cite{KimPark1998, ErnGuermond2006, EISENTRAGER20202289, GlimmConditionNbrNonlocal, Kannan2014, Chavent2003,ErnGuermond2006, FengLewis2018, Astuto2023, Hong2024, Arbogast1995, Mabuza2018Local} and references therein. However, the vast majority of these studies have been applied to local PDEs. Moreover, most of these studies focus on local parabolic and elliptic-type equations \cite{Chavent2003,ErnGuermond2006, FengLewis2018, Astuto2023, Hong2024}, with only a few focusing on local hyperbolic/transport-type equations \cite{Arbogast1995, Mabuza2018Local}.

Non-local models of diffusion-advection or reaction-diffusion-advection types have been developed over the past two-three decades to describe the long-distance interactions between different components of biological, medical and ecological systems~\cite{LeeMogilner2001, MarguetEftimieLozinski, PainterHillenPotts2024, PalMelnik2025}. For example, such non-local models been developed to describe the collective dynamics of cells in the context of disorders such as cancer (e.g., interactions among cancer cells, and interactions between cells and the tissue micro-environment)~\cite{BitsouniEtAl2018,BitsouniChaplainEftimie2017,MarguetEftimieLozinski,DomschkeTrucuGerischChaplain,EckardtPainterSurulescuZhigun,PainterHillenPotts2024}, or in morphogenesis~\cite{GlimmConditionNbrNonlocal}, or to describe the collective dynamics of individuals in ecological settings (e.g., interactions among individuals, or between individuals and the environment)~\cite{EllefsenRodriguez,PalMelnik2025,MarguetEftimieLozinski,PainterHillenPotts2024,LeeMogilner2001}.

To our knowledge, there are extremely few studies that investigate the condition number of coupled non-local time-varying PDEs; see, for example,~\cite{GlimmConditionNbrNonlocal}, which focuses on a finite difference (not finite element) discretisation of a non-local model describing a medical problem. Therefore, to address this lack of results for non-local PDE models discretised via the finite element method, in this work we aim to study the condition number for a system of coupled non-local reaction-diffusion-advection equations developed in the context of normal and abnormal wound healing, which was previously introduced in~\cite{AdebayoEtAl2023}. The non-local terms appear in the transport (advection) term of the 2D model introduced in~\cite{AdebayoEtAl2023} (and described in Section~\ref{Sec:Model description} below).

The paper is structured as follows. In Section~\ref{sec:cond-1}, we recall some elementary details regarding the condition such as its definition and its dependent on the property of the matrix being considered. In Section~\ref{Sec:Model description}, we describe the mathematical model considered. In Section~\ref{Sec:condSyst}, we derive a bound for the model's condition number of the model and a general discussion of the results. In Section~\ref{Sec:summary}, we conclude with a summary of the results.

\section{Brief overview of the condition number of a non-singular matrix}
\label{sec:cond-1}
The condition number is an important measure that identifies the sensitivity of a system to perturbations in its input, in other words, how much the solution of a system will change with respect to a small change or a small error in its input. Hence, it checks the sensitivity of the solution to round-off errors -- an important aspect when solving large systems via numerous iterations. It is a known fact that a very large condition number in a system leads to a huge loss in the number of significant digits in its solution~\cite{Bathe2014, EISENTRAGER20202289, dalhquist74}.
A system with a low condition number is called \emph{well-conditioned}, while a system with a high condition number is called \emph{ill-conditioned}. \\
In brief, for any square non-singular matrix $\M\in \mathbb{R}^{I\times I}$, the condition number, usually denoted by $k(\M)$, is defined as 
\begin{equation*}
k(\mathcal{M}):= \frac{\sigma_{_{\max}}(\mathcal{M})}{\sigma_{_{\min}}(\mathcal{M})}
\end{equation*}
where $\sigma_{max}(\M)$  and $\sigma_{min}(\M)$ are the maximum and minimum singular values of matrix $\M$, respectively. Finally, for completion, we recall here that the singular values are the eigenvalues of the square root of the matrix $\M^{T}\M$, where where $\M^{T}$ is the adjoint of $\M$. Moreover, for the special case of symmetric positive definite matrices, we have that the the singular values coincide with the eigenvalues of the matrix, and in this case the condition number can be defined as
\begin{equation}
    k(\mathcal{M})=\frac{\Lambda_{_{\max}}(\mathcal{M})}{\Lambda_{_{\min}}(\mathcal{M})},
\end{equation}
where $\Lambda_{_{\max}}(\mathcal{M}), \Lambda_{_{\min}}(\mathcal{M})$ are the maximum and minimum eigenvalues, respectively.

\section{Description of non-local PDE model}
\label{Sec:Model description}
In this study, we consider the following non-local model for wound healing, which was previously introduced in~\cite{AdebayoEtAl2023} and that focuses on the interactions between a growth factor ($g(\mathbf{x},t)$), fibroblasts ($f(\mathbf{x},t)$), macrophages ($m(\mathbf{x},t)$) and the extracellular matrix ($e(\mathbf{x},t)$):
\begin{subequations}
\label{EqModel}
\begin{align}
\frac{\partial g}{\partial t} &= D_{_{g}}\Delta g - \lambda_{_{g}} g+ p^{^f}_{_g} f + p^{^m}_{_g} m,\label{EqModel;a}\\ 
\frac{\partial f}{\partial t} &= \nabla \cdot (D_{_{f}}\nabla f - \mu_{f}f A_{_{f}}\left[{u}\right]) -\lambda_{_{f}} f+p_{_{f}}{g}\; f\;(1-\rho(u)), \label{EqModel;b}\\
\frac{\partial m}{\partial t} &= \nabla \cdot (D_{_{m}}\nabla m - \mu_{m}m A_{_{m}}\left[{u}\right]) - \lambda_{_{m}} m +p_{_{m}}{g}\; m\;(1-\rho(u)), \label{EqModel;c}\\
\frac{\partial e}{\partial t} &= - e (\alpha_{_{f}}f+\alpha_{_{m}}m+{\alpha_{e}}) +p_{_{e}}{f} e (1-\rho(u))+{e_{c}}.\label{EqModel;d}
\end{align}
\end{subequations}
 The equations above rely on the following assumptions (detailed in~\cite{AdebayoEtAl2023}): \begin{enumerate}
 \item[(a)] The growth factor $g$ diffuses with a coefficient $D_{g}$, decays at a rate $\lambda_{g}$, and is produced by fibroblasts and macrophages at rates $p_{g}^{f}$ and $p_{g}^{m}$, respectively; 
 \item[(b)] The fibroblasts $f$ exhibit random motion with a coefficient $D_{f}$ and directed movement with a coefficient $\mu_{f}$, undergo apoptosis at a constant rate $\lambda_{f}$, and proliferate at a rate $p_{f}$ in response to the growth factor; 
 \item[(c)] The macrophages $m$ move randomly with a coefficient $D_{m}$ and directionally with a coefficient $\mu_{m}$, die at a constant rate $\lambda_{m}$, and proliferate at a rate $p_{m}$ in the presence of the growth factor; 
 \item[(d)] The extracellular matrix (ECM) $e$ is remodeled by fibroblasts at a rate $p_{e}$ and degraded by enzymes produced by fibroblasts and macrophages at rates $\alpha_{f}$ and $\alpha_{m}$, respectively~\cite{ZhaoEtAl2022}. Additionally, it is assumed that ECM components can be secreted at a low rate $e_{c}$ by other cells in the environment (e.g., endothelial cells~\cite{XueEtAl2015}; not accounted for here) and degraded at a low rate $\alpha_{e}$ by enzymes produced by other cells (e.g., keratinocytes and endothelial cells~\cite{XueEtAl2015}; not accounted for here).
 \end{enumerate}

Lastly, the system is defined on a 2D spatial domain with zero-flux boundary conditions \textit{i.e.,}

\begin{equation}
\spta{\nabla g, n} = 0,\quad \text{ on } (0, \infty) \times \partial\Omega,
\label{EqNonlocalFlux:ExisUniq}
\end{equation}
{and}
\begin{equation}
\spta{D_{_{f}}\nabla f - \mu_{f}f A_{_{f}}\left[{u}\right], n} =\spta{D_{_{m}}\nabla m - \mu_{m}m A_{_{m}}\left[{u}\right], n}  =0,\quad \text{ on } (0, \infty) \times \partial\Omega ,
\label{EqBC:ExisUniq}
\end{equation}
{where $n$ is the outward unit normal vector to $\partial \Omega$, and $ \spta{ \cdot,\cdot }$ is the usual Euclidean scalar product in finite dimension.
To complete the model description, we assume the following initial conditions of the form:}
\bequ
{\begin{array}{ll}
          {g(\mathbf{x},0) = g_{_0}(\mathbf{x})\ge 0},& \quad  {f(\mathbf{x},0) = f_{_0}(\mathbf{x})\ge 0,}\\[0.2cm] 
         {m(\mathbf{x},0) = m_{_0}(\mathbf{x})\ge 0,}& \quad  {e(\mathbf{x},0) = e_{_0}(\mathbf{x})\ge 0,} \quad{\text{ for all } \mathbf{x}\in\Omega.}
\end{array}}
\label{EqIC:ExisUniq}
\eequ
\noindent Note that in \cite{AdebayoEtAl2023}, it was assumed that the ECM can be produced and degraded only by the macrophages and the fibroblasts considered in this system (i.e., $\alpha_{e}=e_{c}=0$). Here, we take the same approach as in \cite{AdebayoTrucuEftimie2025} and we assume that $\alpha_{e},e_{c} \neq 0$.
Moreover, in equations \eqref{EqModel}, the term $\rho(u): = w_{_{g}} g + w_{_{f}} f + w_{_{m}} m + w_{_{e}} e$ represents the combined volume fraction of space occupied by the various components of the system. Here, $w_{_{g}}, w_{_{f}}, w_{_{m}}, w_{_{e}}>0$ are coefficients that indicate how much space is occupied by the growth factor, fibroblasts, macrophages, and ECM, respectively. The functions ${A_{i}[u],i\in{f,m}}$ that appear in equations \eqref{EqModel} describe the adhesive interactions among cells, and between cells and the ECM within a specific sensing region. This sensing region is denoted by $\textbf{B}_{_{\|\cdot\|_{_{\infty}}}}\!\!(\textbf{0},R)$, which represents the area surrounding a point $\textbf{x}$, extending outwards by a distance $R$. These interactions are examined at each point in time $t\in[0,T]$. The region $\textbf{B}_{_{\|\cdot\|_{_{\infty}}}}\!\!(\textbf{0},R)$ is essentially a square with a side length of $2R$, centered at the origin.
In simple terms, the functions $A_{{i}}[u]$ capture how cells interact with each other and with the ECM within this defined region. These interactions are depicted by the following non-local terms, which are described in detail in~\cite{AdebayoEtAl2023}.
\begin{equation}    
        {A_{i}[u](\textbf{x},t)} = \frac{1}{R}\,\,{\int\limits_{\textbf{B}_{_{\|\cdot\|_{_{\infty}}}}\!\!(\textbf{0},R)}}\!\!\!\mathrm{K}(\|\mathbf{y}\|_{_{2}})  \mathbf{n}(\mathbf{y}){\left((1-\rho(u)\right)^{+}\,\mathbf{\Gamma}_{_i}(u))}(\mathbf{x}+\mathbf{y},t)\, \text{d}\mathbf{y}, \quad\quad {i\in\{f,m\}}.
   \label{non-local term}
\end{equation}
The term $\left(1-\rho(u)\right)^{+}:= \max\{0, 1-\rho(u)\}$ is used to prevent overcrowding by ensuring that the volume fraction does not exceed a maximum value. This function returns the positive difference between 1 and $\rho(u)$, or 0 if $\rho(u)$ is greater than or equal to 1.
For any point $y$ within the region $\textbf{B}_{_{\|\cdot\|_{_{\infty}}}}\!\!(\textbf{0},R)$, the vector $\mathbf{n(y)}$ represents the unit radial vector starting from $\mathbf{x}$ and pointing towards the location $\mathbf{x+y}$. This vector is defined for all $\mathbf{y}$ in the region $\textbf{B}_{_{\|\cdot\|_{_{\infty}}}}\!\!(\textbf{x},R)$ as follows:
\begin{equation}
\mathbf{n(y)} :=  \begin{cases}
         \frac{\mathbf{y}}{\|\mathbf{y}\|_{_{2}}}, & \text{if } \mathbf{y}\in [-R,R]^{^{2}}\setminus \{(0, 0)\}\\
        (0, 0), & \text{otherwise, }\end{cases}
        \label{surf:normal:def}
\end{equation}
where $\textbf{x}:=(x_{_1}, x_{_2})$,  $\textbf{y}:=(y_{_1}, y_{_2})$ and  $\|\mathbf{y}\|_{_{2}} = \sqrt{y_{_1}^{^2}+y_{_2}^{^2}}$. Finally, $\mathbf{\Gamma}_{_{i}} {(u)}: \textbf{B}_{_{\|\|_{_{\infty}}}}(\mathbf{x},R)\to~\R$ 
cumulates the strengths of the adhesive junctions formed over the sensing region at time $t\in[0,T]$ and are given as
\begin{equation*}
\label{Eq:AdhesionType}
    {\mathbf{\Gamma}_{_{i}}(u) =\sum_{j\in\{f, m, e\}}\mathbf{S}_{_{ij}}(u)\,j,\quad \quad i\in\{f, m\}, \,j\in\{f, m, e\}.} 
\end{equation*}
 {This function describes the adhesive interactions between cell population $i$ and cell population $j$, $\forall\, i, j\in\{f,m\}$, as well as the adhesive interactions between each cell population and the ECM. Thus for a compact notation, we have $\mathbf{S}_{_{ij}}$ given by}
\begin{equation*} 
    {\mathbf{S}_{_{ij}}(u):=\mathbf{S}_{_{ij}}^{^{\max}}\frac{e + g}{1 + e + g}, \quad i\in\{f, m\},\, j\in\{f, m, e\}.}
    \label{AdhesionFunct2}
\end{equation*}
{To account for the dependence of the strength of the cell-cell/cell-matrix interactions on the spatial distribution of cells and the spatial distribution of ECM inside the sensing region, we were inspired here by the experimental evidence regarding the distribution of cell-cell sensing distance observed in cancerous tissues \cite{Parra2021}. As in \cite{AdebayoEtAl2023}, the main kernel that we consider is the following radial-dependent kernel~\cite{BitsouniEtAl2018, AdebayoEtAl2023}}:
\begin{equation}
        \mathrm{K}(\|\mathbf{y}\|_{_{2}}) = \displaystyle \frac{\|\mathbf{y}\|_{_{2}}}{2\pi\sigma^{^{2}}}\,e^{^{-{\|\mathbf{y}\|_{_{2}}^{^{2}}\big/2\sigma^{^{2}}}}}.
        \label{kernel:Gaussian}
\end{equation}
This enables distance-induced weighting for these adhesion interactions on the sensing region \textbf{B}(\textbf{x}, R) between cells or ECM distributed at spatial location \textbf{x+y} and the cells distributed at \textbf{x}. 

An example of the spatio-temporal dynamics obtained with model~\eqref{EqModel} is shown Figure~\ref{fig:model_dynamic} (Appendix~\ref{Sect:Appendix3}), where one can see a repair of an initial linear wound (i.e., cut) in the 2D tissue domain, for the specific parameter values listed in the Appendix. 
\section{Brief discussion on numerical discretization}

Here we describe the approaches we considered to discretize numerically model~\eqref{EqModel} (for the simulations shown in Figure~\ref{fig:model_dynamic}, and for the analysis presented below). First, we write the model equations together with the boundary conditions using the following compact notation:
\begin{equation}
    \frac{\partial u}{\partial t} = \mathrm{E}(u), \quad u(\mathbf{x},0) = u_0, \quad \frac{\partial u}{\partial n}\bigg|_{\partial \Omega} = 0, \label{EqModelCompact}
\end{equation}
where  \(u = (u_{1}, u_{2},u_{3},u_{4}) = (g, f, m, e)^T \in (H^1(\Omega; [0, T]))^4 \), over \( \Omega \times [0, T] \),
and \(\mathrm{E}(\cdot):=(\E_{u_{1}}(\cdot), \E_{u_{2}}(\cdot),\) \(\E_{u_{3}}(\cdot), \E_{u_{4}}(\cdot))^{T}\) is the right-hand-side spatial operator of the model. As usual, we obtain the weak formulation of model \eqref{EqModelCompact} by multiplying the equations by the test function \(v
\in\mathcal{C}_0^\infty(\Omega)\) and integrating over \(\Omega\) \textit{i.e.,}
\begin{equation*}
    \int_{\Omega} \frac{\partial u}{\partial t} v\, d\mathbf{x} = \int_{\Omega} \mathrm{E}(u) v\, d\mathbf{x}, \quad \forall v \in \mathcal{C}_0^\infty(\Omega).
\end{equation*}
Furthermore, for space discretisation of variables $u_{i}$ we use P2 basis functions $\{\psi_{\tau}(\cdot) \}_{\tau=0}^{l}$ corresponding to a rectangular grid of $l+1$ uniformly distributed nodes. Thus, on a square mesh, the solution components are approximated by
\begin{equation}
    \tilde{u}_i(\mathbf{x}, t) = \sum_\tau c^{u_i}_\tau(t)\, \psi_\tau(\mathbf{x}).
    \label{Eq:space-discetization}
\end{equation}
Finally, we discretize equations \eqref{EqModelCompact} in time using the splitted backward Euler given as 
\begin{equation}
    \int_{\Omega} \frac{u_{i}^{N+1} - u_{i}^N}{\Delta t} v\, d\mathbf{x} = \int_{\Omega} \E_{u_{i}}(u_{1}^{N},\dots, u_{i-1}^{N}, u_{i}^{N+1}, u_{i+1}^{N},\dots, u_{4}^{N}) v\, d\mathbf{x}.
    \label{Eq:time-discetization}
\end{equation}
To obtain the fully discretized system, we project onto the basis \(\psi_{j}\) and apply the time-marching scheme, which results into a (non-)linear system that has to be solved at each time step. This is stated in full detail in Appendix~\ref{appendix-full-discretization}.

\section{Condition number of the system}
\label{Sec:condSyst}
In this section, we estimate the condition number of the spatio-temporal discretized version of model \eqref{EqModel} at time \(N+1\) (i.e., system \eqref{Eq:time-discetization}) by first estimating the condition number of each of the constituting equations (\textit{i.e}., for variables \(g, f, m\) and \(e\){)}. 
For this, we employ the Rayleigh quotient to estimate the maximum and minimum eigenvalues (whose ratio is the condition number; see Section~\ref{sec:cond-1}), since the singular values coincide with the spectrum of the resulting (symmetric positive-definite) matrix. Note that, for simplicity, throughout this section we denote $u(\mathbf{x})$ as $u$, and $\psi_{_{(\cdot)}}(\mathbf{x})$ as $\psi_{_{(\cdot)}}$.  Also, from now and till the end of the study, we shall use the symbol $u$ for the finite element approximation at time $t_{N+1}, \: u=\{g, f, m ,e\}$ and consequentially, $\mathbf{c}^{u}\in\mathbb{R}^{l}$ for the approximate nodal values of $u$.

\paragraph{Condition number for the $g$-equation.} 
To estimate the condition number of the $g$-equation (i.e., equation \eqref{EqModel;a}), we derive the weak form of this equation by discretizing in time using backward Euler, multiplying by test function $v_{g}$, and integrating over \(\Omega\) (while also applying integration by parts and imposing the boundary condition in~\eqref{EqNonlocalFlux:ExisUniq}) to obtain
\begin{align}
    \left(\frac{1}{\Delta t}+\lambda_{_{g}}\right)\int\limits_{\Omega}\{g\,v_{_{g}} + D_{_{g}} \spt{\nabla g,\nabla v_{_{g}}}\}\text{d}\mathbf{x}=\int\limits_{\Omega}\{\frac{1}{\Delta t}g^{N} +  p_{_{g}}^{f} f^{N} + p_{_{g}}^{m}m^{ N}\}v_{_{g}}\text{d}\mathbf{x}
    \quad \forall\; v_{_{g}}.
    \label{Eq:weakform-gf}
\end{align}
We proceed to further discretize in space using finite element methods and write the above equation in the form
\begin{align*}
    \mathbf{A}_{_{g}} \mathbf{c}^{g} = \mathbf{F}_{g}
\end{align*}
and then proceed to estimate the condition number $k(\mathbf{A}_{_{g}})$ for the matrix $\mathbf{A}_{_{g}}$. After spatial discretization, \eqref{Eq:weakform-gf} recasts as
\begin{align*}
        \sum_{\tau=1}^{l}&\Bigg(\left(\frac{1}{\Delta t}+\lambda_{_{g}}\right)\int\limits_{\Omega}\psi_{\tau}\psi_{o}\textbf{d}\mathbf{x} + D_{_{g}} \int_{\Omega}\spt{\nabla \psi_\tau,\nabla \psi_{o}}\,\text{d}\mathbf{x}\Bigg)\textbf{c}^{g}\\
        &= \sum_{\tau=1}^{l}\left(\frac{1}{\Delta t}{c}^{g, N} +  p_{_{g}}^f{c}^{f, N} + p_{_{g}}^m{c}^{m, N}\right)\int\limits_{\Omega}\psi_{\tau}\psi_{o}\textbf{d}\mathbf{x}
    \quad \forall\; o\in \{1, \ldots, l\},
\end{align*}
and in a compact form as
\begin{equation*}
    \left(\left(\frac{1}{\Delta t}+\lambda_{_{g}}\right)\textbf{M} + D_{_{g}}\textbf{K}\right)\textbf{c}^{g} = \,\textbf{M}\left(\frac{1}{\Delta t}\textbf{c}^{g, N} +  p_{_{g}}^f \textbf{c}^{f, N} + p_{_{g}}^m\textbf{c}^{m, N}\right),
\label{eq:g_r1}
\end{equation*}
where 
\begin{align*}
    \textbf{A}_{_{g}} := \left(\frac{1}{\Delta t}+\lambda_{_{g}}\right)\textbf{M} + D_{_{g}}\textbf{K},\quad \text{and} \quad \mathbf{F}_{g} := \left(\frac{1}{\Delta t}\textbf{c}^{g, N} +  p_{_{g}}^f \textbf{c}^{f, N} + p_{_{g}}^m\textbf{c}^{m, N}\right)
\end{align*}
with $\textbf{M}$ the mass matrix, and $\textbf{K}$ the stiffness matrix. 
We recall that $g\approx\Tilde{g}^{}(x,t)=~\sum_{\tau = 1}^{l}c_{\tau}^{g}(t)\psi_{\tau}(x)$, 
then following the same argument as in Section 2.2 of~\cite{Johnson1987}, we derive the quadratic form of the classical mass matrix \textbf{M}:
\begin{align}
    \|\Tilde{g}\|^2 =\int_{\Omega} \spta {\Tilde{g}^{},\Tilde{g}} \,\,\text{d}\mathbf{x} &= \int_{\Omega} \spta{\sum_{\tau = 1}^{l}c_{\tau}^{g}\psi_{\tau}, \sum_{\tau = 1}^{l}c_{\tau}^{g}\psi_{\tau}} \,\, d\textbf{x} \nonumber\\
    &=\sum_{\tau = 1}^{l}c_{\tau}^{g}\int_{\Omega}\spt{\psi_{\tau}, \psi_{\tau}} \,d\textbf{x}\,\, c_{\tau}^{g}\nonumber\\
    &=\spt{\mathbf{M}\textbf{c}^{g}, \mathbf{c}^{g}}.
    \label{Eq:quadForm-MassMatrix}
\end{align}
We can follow the same argument for the classical stiffness matrix, $\textbf{K}$ \textit{i.e}., 
\begin{align}
   \|\nabla \Tilde{g}\|^2 =\int_{\Omega} \spta{\nabla \Tilde{g},\nabla \Tilde{g}} \,\,\text{d}\mathbf{x} &= \int_{\Omega} \spta{\sum_{\tau = 1}^{l}c_{\tau}^{g}\nabla \psi_{\tau}, \sum_{\tau = 1}^{l}c_{\tau}^{g}\nabla\psi_{\tau}} d\textbf{x} \nonumber\\
    &=\sum_{\tau = 1}^{l}c_{\tau}\int_{\Omega}\spt{\nabla \psi_{\tau},\nabla \psi_{\tau}}\,d\textbf{x}\,\, c_{\tau}^{g}\nonumber\\
    &=\spt{\mathbf{K}\mathbf{c}^{g},\textbf{c}^{g}}.
    \label{Eq:quadForm-StiffMatrix}
\end{align}
To proceed with obtaining the Rayleigh quotient, we assume that the family $\{T_{_{h}}\}$ of triangulations $T_{_{h}} = \{K\}$ satisfies the following condition~\cite{Johnson1987}:\\
\emph{There exist $\eta_{1}, \eta_{_{2}}>0$ independent of $h:=\underset{K \in T_{_{h}}}{\max} h_{_{K}}$, such that
\begin{itemize}
\item $T_{_{h}}$ is quasi-uniform, \textit{i.e}.,
\begin{equation}
     h_{_{K}}\ge \eta_{1} h,
     \label{Eq:h_kBound}
\end{equation}
\item all elements $K$ in the triangulation are well-shaped (\textit{i.e}., they are not too elongated or highly distorted)
\begin{equation} 
\frac{\varrho_{_{K}}}{h_{_{K}}}\ge \eta_{_{2}},
 \label{Eq:varrho_k-h_kBound}
\end{equation}
where $h_{_{K}}$ is the longest side of $K$ and $\varrho_{_{K}}$ is the diameter of the circle inscribed in $K$. 
\end{itemize}
}
Further, we define the space $V = \{g\in H^{^{1}}(\Omega): \spta{\nabla g , n }=0\}$ and we note that our approximations $\Tilde{g}$ are within this space. 
We use existing results~\cite{Johnson1987} on the bounds for the mass matrix and stiffness matrix in the lemma below.
\begin{lemma}
\label{lemma:mass-mat:stiff-mat}
    There exist constants $\zeta^{^{g}}_{1}$ and $\zeta^{^{g}}_{_{2}}$ depending on $\eta_{1}$ and $\eta_{_{2}}$ defined in \eqref{Eq:h_kBound} and \eqref{Eq:varrho_k-h_kBound}, respectively, such that
    \begin{align}
        \zeta^{^{g}}_{1}h^{^{2}}\nortwo{\mathbf{c}^{g}}^{^{2}}\le \|\Tilde{g}\|^{^{2}}_{L^{2}(\Omega)} \le  \zeta^{^{g}}_{_{2}}h^{^{2}}\|\mathbf{c}^{g}\|^{^{2}}_{_{2}}\label{InEq:mm_bound-g}\\
       \|\nabla \Tilde{g}\|^{^{2}}_{L^{2}(\Omega)}\equiv \int_{\Omega} \nortwo{\nabla \Tilde{g}}^{^{2}}\text{d}\mathbf{x}\le \zeta^{^{g}}_{_{2}}h^{^{-2}}\|\Tilde{g}\|^{^{2}}_{L^{2}(\Omega)}\label{InEq:sm_bound-g},
    \end{align}
    where $\nortwo{\cdot}$  is the usual Euclidean norm.
\end{lemma}
\begin{proof}
   See~\cite{Johnson1987}.
\end{proof}
Since the desired matrix (\textit{i.e}., $\mathbf{A}_{_{g}}$) is a linear combination of the matrices $\mathbf{M}$ and $\mathbf{K}$, we obtain its quadratic form also by a linear combination of $\|\Tilde{g}\|^{^{2}}_{L^{2}(\Omega)}$ and $\|\nabla \Tilde{g}\|^{^{2}}_{L^{2}(\Omega)}$. First, we split \eqref{InEq:mm_bound-g} into upper and lower bounds and multiply both by $\left(\frac{1}{\Delta t}+\lambda_{_{g}}\right)$:
\begin{align}
    \left(\frac{1}{\Delta t}+\lambda_{_{g}}\right)\zeta^{^{g}}_{1}h^{^{2}}\|\mathbf{c}^{g}\|^{^{2}}_{_{2}}\label{Ineq:lowerbnd-massmatrix}&\le \left(\frac{1}{\Delta t}+\lambda_{_{g}}\right)\spt{\mathbf{M}\textbf{c}^{g}, \mathbf{c}^{g}},\\
    \left(\frac{1}{\Delta t}+\lambda_{_{g}}\right)\spt{\mathbf{M}\textbf{c}^{g},\mathbf{c}^{g}}&\le \left(\frac{1}{\Delta t}+\lambda_{_{g}}\right)\zeta^{^{g}}_{_{2}} h^{^{2}}\|\mathbf{c}^{g}\|^{^{2}}_{_{2}},\label{Ineq:upperbnd-massmatrix}
\end{align}
where we have used inequality~\eqref{InEq:mm_bound-g}  with the identity in \eqref{Eq:quadForm-MassMatrix}. We point out that the coercivity property does not hold for the bilinear form $\int_{\Omega}\spta{ \nabla \Tilde{g},\nabla \Tilde{g} } \text{d}\mathbf{x}$ in \(V\) for homogeneous Neumman boundary conditions, because $\nabla \Tilde{g}  = 0$ for $\Tilde{g} = $ constant functions. We have $\int_{\Omega}\spta{ \nabla \Tilde{g},\nabla \Tilde{g} } \text{d}\mathbf{x}\ge 0$, and hence multiplying this and \eqref{InEq:sm_bound-g} by $D_{_{g}}>0$ leads to the following inequalities
\begin{align}
0&\le D_{_{g}}\spt{\mathbf{K}\textbf{c}^{g},\mathbf{c}^{g}}\label{Ineq:lowerbnd-stiffmatrix},\\
    D_{_{g}}\spta{\mathbf{K}\textbf{c}^{g}, \mathbf{c}^{g}}&\le D_{_{g}}\zeta^{^{g}}_{_{2}}\|\mathbf{c}^{g}\|^{^{2}}_{_{2}}\label{Ineq:upperbnd-stiffmatrix}.
\end{align}
Adding \eqref{Ineq:lowerbnd-massmatrix} to \eqref{Ineq:lowerbnd-stiffmatrix}, and \eqref{Ineq:upperbnd-massmatrix} to \eqref{Ineq:upperbnd-stiffmatrix}, we obtain 
\begin{align}
 \left(\frac{1}{\Delta t}+\lambda_{_{g}}\right)\zeta^{^{g}}_{1} h^{^{2}}\|\mathbf{c}^{g}\|^{^{2}}_{_{2}}&\le  \spta{ \left(\left(\frac{1}{\Delta t}+\lambda_{_{g}}\right)\mathbf{M} +  D_{_{g}} \mathbf{K}\right)\mathbf{c}^{g}, \mathbf{c}^{g}}\label{Ineq:A_lowbnd},\\
  \spta{\left(\left(\frac{1}{\Delta t}+\lambda_{_{g}}\right)\mathbf{M} +  D_{_{g}} \mathbf{K}\right)\mathbf{c}^{g}, \mathbf{c}^{g}} &\le \left(\left(\frac{1}{\Delta t}+\lambda_{_{g}}\right)\zeta^{^{g}}_{_{2}}h^{^{2}} + D_{_{g}}\zeta^{^{g}}_{_{2}} \right)\|\mathbf{c}^{g}\|^{^{2}}_{_{2}}\label{Ineq:A_upbnd}.
\end{align}
Then, from \eqref{Ineq:A_lowbnd} and \eqref{Ineq:A_upbnd}, we obtain the following lower and upper bounds for the singular values of $\mathbf{A}_{_{g}}$, namely:
\begin{align*}
    &\frac{\spta{ \mathbf{A}_{{g}}\mathbf{c}^{g}, \mathbf{c}^{g}}}{\|\mathbf{c}^{g}\|^{^{2}}_{_{2}}} = \frac{\spta{ \left(\left(\frac{1}{\Delta t}+\lambda_{_{g}}\right)\mathbf{M} +  D_{_{g}} \mathbf{K}\right)\mathbf{c}^{g}, \mathbf{c}^{g}}}{\|\mathbf{c}^{g}\|^{^{2}}_{_{2}}}\ge\left(\frac{1}{\Delta t}+\lambda_{_{g}}\right)\zeta^{^{g}}_{1} h^{^{2}} ,\\
     &\frac{\spta{ \mathbf{A}_{{g}}\mathbf{c}^{g}, \mathbf{c}^{g}}}{\|\mathbf{c}^{g}\|^{^{2}}_{_{2}}} = \frac{\spta{ \left(\left(\frac{1}{\Delta t}+\lambda_{_{g}}\right)\mathbf{M} +  D_{_{g}} \mathbf{K}\right)\mathbf{c}^{g}, \mathbf{c}^{g}}}{\|\mathbf{c}^{g}\|^{^{2}}_{_{2}}}\le \left(\left(\frac{1}{\Delta t}+\lambda_{_{g}}\right)\zeta^{^{g}}_{_{2}} h^{^{2}} + D_{_{g}}\zeta^{^{g}}_{_{2}} \right),
\end{align*}
$\forall \; \mathbf{c}^{g}\in\mathbb{R}^{l}$. Since, $\mathbf{M}$ and $\mathbf{K}$ are symmetric and positive definite, then $\mathbf{A}_{_{g}}$ is also symmetric and positive definite and hence the spectrum of  $\mathbf{A}_{_{g}}$ and its singular values coincide. Therefore, the condition number can now be expressed in terms of the eigenvalues \textit{i.e}., $\frac{\Lambda_{\max}}{\Lambda_{\text{min}}}$. This shows that there exists constants $\zeta^{^{g}}_{1}$ and $\zeta^{^{g}}_{_{2}}$ such that the condition number is given by
\begin{equation}
    {k}(\mathbf{A}_{{g}}) = \frac{\Lambda_{\max}}{\Lambda_{\text{min}}}\le \frac{\left(\frac{1}{\Delta t}+\lambda_{_{g}}\right)\zeta^{^{g}}_{1}h^{^{2}} +D_{_{g}}\zeta^{^{g}}_{_{2}}}{\left(\frac{1}{\Delta t}+\lambda_{_{g}}\right)\zeta^{^{g}}_{_{2}}h^{^{2}}} = 1 +\frac{D_{_{g}}\zeta^{^{g}}_{_{2}}}{\left(\frac{1}{\Delta t}+\lambda_{_{g}}\right)\zeta^{^{g}}_{_{2}}} h^{^{-2}}.
    \label{ConditionNumber_g}
\end{equation}
\paragraph{Condition number for the $e$-equation.} For the equation corresponding to the ECM dynamics, we discretize it using backward Euler, multiply it by a test function $v_{e}$ and integrate over \(\Omega\) to obtain the weak form
\begin{align}
    \int_{\Omega}\! \bigg[  e\!\left(\!\frac{1}{\Delta t} + \alpha_{_{f}}f^{N}\!\!+ \alpha_{_{m}}m^{N}\! \!+ \!\alpha_{_{e}}\!\!\right) 
    - p_{_{e}}f^{N}\!e\!\left(\!1 \!-\! w_{_{g}}g^{N}\!-\! w_{_{f}} f^{N}\!-\! w_{_{m}} m^{N}\!-\! w_{_{e}} e\right)\!\! \bigg] v_{_{e}} \text{d}\textbf{x}\nonumber\\
    = \int_{\Omega}\left(e_{_{c}}+ \frac{e^{N}}{\Delta t}\right)v_{_{e}}\,\text{d}\textbf{x}.
    \label{Eq:ECM_restated}
\end{align}
We see that the second term of the integrand on the left-hand-side (which we denote as $T^{e}_{_{2}}$) is non-linear in the unknown $e$. For that reason, we linearize it with respect to the unknown values $e$ evaluated at the known (\textit{i.e}., at time $N$) values of the other components $g, f,$ and $m$: 
\begin{equation}
    \frac{d T^{e}_{_{2}}}{d e} = 2 p_{_{e}}f^{N}w_{_{e}} e - p_{_{e}}f^{N}\left(1 \!-\! w_{_{g}}g^{N}\!-\! w_{_{f}} f^{N}\!-\! w_{_{m}} m^{N}\right).
    \label{Eq:linearizing_logistic_e}
\end{equation}
Now replacing this second term ($T^{e}_{_{2}}$) in the left-hand-side of \eqref{Eq:ECM_restated} with its linear approximation given by the derivative in \eqref{Eq:linearizing_logistic_e}, we get
\begin{align*}
    \int_{\Omega}\!e\!\left(\!\!\frac{1}{\Delta t}\! + \!\alpha_{_{f}}f^{N}\!\!\!+ \alpha_{_{m}}m^{N} \!\!+\! \alpha_{_{e}}\!\! +  2 p_{_{e}}f^{N}w_{_{e}}\!\!\right)\!v_{_{e}}\text{d}\textbf{x}
    = \!\int_{\Omega}\!\left(\!\!\left(e_{_{c}}\!+ \!\frac{e^{N}}{\Delta t}\!\!\right)\!\!+ \!p_{_{e}}f^{N}\!\!\left(\!1 \!-\! w_{_{g}}g^{N}\!\!-\! w_{_{f}} f^{N}\!\!-\! w_{_{m}} m^{N}\!\right)\!\!\right)\!v_{_{e}}\text{d}\textbf{x},
\end{align*}
which after discretizing in space becomes
\begin{align*}
     &\sum_{\tau = 1}^{l}\int_{\Omega}c_{\tau}^{^{e}}\psi_{\tau}\left(\frac{1}{\Delta t} + \sum_{q = 1}^{l}\psi_{q}\left(\left(\alpha_{_{f}} +  2 p_{_{e}}w_{_{e}}\right)c_{q}^{f,N}+ \alpha_{_{m}}c_{q}^{m, N} \right)+ \alpha_{_{e}}\right)\psi_{o}\text{d}\textbf{x} \nonumber\\
     &= \int\limits_{\Omega}\!\sum_{\tau = 1}^{l} \!\Bigg[ p_{_{e}} c_{\tau}^{f, N}\psi_{\tau}\!\! \left(\!\!1 -\! \sum_{s = 1}^{l}\psi_{s} \!\!\left(w_{_{g}}c_{s}^{g,N}\!\! +w_{_{f}}c_{s}^{f,N} \!\!+w_{_{m}}c_{s}^{m,N} \right)\!\!\right)\!\! + e_{_{c}} \!+\!\frac{1}{\Delta t} c_{\tau}^{e,N}\psi_{\tau}\!\Bigg]\psi_{o}\text{d}\textbf{x}.
\end{align*}
This is already expressed in the general form 
$\mathbf{A}_{e} \mathbf{c}^{e} = \mathbf{F}_{e}$, where \( \mathbf{A}_{e}\in\mathbb{R}^{l\times l} \) with entries
\begin{equation*}
    \left(\mathbf{A}_{e}\right)_{\tau o} = \int_{\Omega}\psi_{\tau}\psi_{o}\left(\frac{1}{\Delta t} + \sum_{q = 1}^{l}\psi_{q}\left(\left(\alpha_{_{f}} +  2 p_{_{e}}w_{_{e}}\right)c_{q}^{f,N}+ \alpha_{_{m}}c_{q}^{m, N} \right)+ \alpha_{_{e}}\right)\text{d}\textbf{x},
\end{equation*}
and
\begin{align*}
    \mathbf{F}_{e} = \int\limits_{\Omega}\!\sum_{\tau = 1}^{l} \!\Bigg[ p_{_{e}} c_{\tau}^{f, N}\psi_{\tau}\!\! \left(\!\!1 -\! \sum_{s = 1}^{l}\psi_{s} \!\!\left(w_{_{g}}c_{s}^{g,N}\!\! +w_{_{f}}c_{s}^{f,N} \!\!+w_{_{m}}c_{s}^{m,N} \right)\!\!\right)\!\! + e_{_{c}} \!+\!\frac{1}{\Delta t} c_{\tau}^{e,N}\psi_{\tau}\!\Bigg]\psi_{o}\text{d}\textbf{x},\\
    \,\,\forall o\in \{1,\ldots, l\}.
\end{align*}
Note that we can rewrite the entries of \(\mathbf{A}_{e}\) as follows:
\begin{align*}
    (\mathbf{A}_{e})_{_{\tau o}} &= \int_{\Omega}\psi_{\tau} \psi_{o}\left(\frac{1}{\Delta t} + \sum_{q = 1}^{l}\psi_{q}\left(\left(\alpha_{_{f}} +  2 p_{_{e}}w_{_{e}}\right)c_{q}^{f,N}+ \alpha_{_{m}}c_{q}^{m, N} \right)+ \alpha_{_{e}}\right)\text{d}\textbf{x}\nonumber\\
    & = \left(\frac{1}{\Delta t} +\alpha_{_{e}}\right)\!\!\!\int_{\Omega}\psi_{\tau}\psi_{o} \text{d}\textbf{x} + \sum_{q = 1}^{l} \left(\left(\alpha_{_{f}} +  2 p_{_{e}}w_{_{e}}\right)c_{q}^{f,N}+ \alpha_{_{m}}c_{q}^{m, N} \right)\int_{\Omega}\psi_{\tau}\psi_{q}\psi_{o} \text{d}\textbf{x}.
\end{align*}
This means that the matrix \( \mathbf{A}_{e}\in \mathbb{R}^{^{l\times l}}\) can now be expressed as
\begin{align}
    \mathbf{A}_{e} = \left(\frac{1}{\Delta t} +\alpha_{_{e}}\right)\mathbf{M} + \left(\alpha_{_{f}} +  2 p_{_{e}}w_{_{e}}\right)\mathbf{T}\left(c^{f, N}\right)+ \alpha_{_{m}}\mathbf{T}\left(c^{m, N}\right),
     \label{Eq:leftMatrixECM}
\end{align}
where \(\mathbf{M} = 
    \begin{bmatrix}
        \int\limits_{\Omega}  \psi_{1} \psi_{1}d\mathbf{x} &\ldots & \int\limits_{\Omega} \psi_{1} \psi_{l}d\mathbf{x}\\
        \vdots & & \vdots\\
         \int_{\Omega} \psi_{l} \psi_{1}d\mathbf{x} &\ldots & \int\limits_{\Omega} \psi_{l} \psi_{l}d\mathbf{x}
    \end{bmatrix}\), and \\[2.0ex]
    \\
$\mathbf{T}\left(c^{i, N}\right) =  \begin{bmatrix}
        \sum_{q}^{l} c_{q}^{i, N} \int\limits_{\Omega} \psi_{q} \psi_{1} \psi_{1}d\mathbf{x} &\ldots &  \sum_{q}^{l} c_{q}^{i, N} \int\limits_{\Omega} \psi_{q} \psi_{1} \psi_{l}d\mathbf{x}\\
        \vdots & & \vdots\\
          \sum_{q}^{l} c_{q}^{i, N} \int_{\Omega} \psi_{q} \psi_{l} \psi_{1}d\mathbf{x} &\ldots &  \sum_{q}^{l} c_{q}^{i, N} \int\limits_{\Omega} \psi_{q} \psi_{l} \psi_{l}d\mathbf{x}\end{bmatrix}, \quad i\in\{f, m\}$.\\[2.0ex] 

Since the terms \(\mathbf{T}_{c^{i, N}}, i\in\{f, m\}\) involve the integral of triple $\psi_{_{(\cdot)}}$, now we derive  the general quadratic form for this kind of terms. For any $\Tilde{e} =\sum_{\tau}^{l}c_{\tau}^{\Tilde{e}}\psi_{\tau} $ and an arbitrary known solution  \(\Tilde{u}^{N} =~\sum_{\tau}^{l}c_{\tau}^{\Tilde{u}, N}\psi_{\tau}\), \(\Tilde{u}\in\{\Tilde{g}, \Tilde{f}, \Tilde{m}, \Tilde{e}\}\) we have
{\footnotesize
\begin{align}
    &\int_{\Omega} \spta{\spta{ \Tilde{e}, \Tilde{e} }, \Tilde{u}^{N}}\; \text{d} \mathbf{x} 
     = \sptb{
    \begin{bmatrix}
        \sptb{
        \begin{bmatrix}
           \int\limits_{\Omega} \psi_{1} \psi_{1} \psi_{1}d\mathbf{x} &\ldots&  \int\limits_{\Omega} \psi_{1} \psi_{1} \psi_{l}d\mathbf{x}\nonumber\\
           \vdots & & \vdots\nonumber\\
           \int\limits_{\Omega} \psi_{1} \psi_{l} \psi_{1}d\mathbf{x} &\ldots & \int\limits_{\Omega} \psi_{1} \psi_{l} \psi_{l}d\mathbf{x}
       \end{bmatrix}
       \mathbf{c}^{e},\mathbf{c}^{e}}\nonumber\\
       \vdots\\
       \sptb{
       \begin{bmatrix}
           \int\limits_{\Omega} \psi_{l} \psi_{1} \psi_{1}d\mathbf{x} &\ldots&  \int\limits_{\Omega} \psi_{l} \psi_{1} \psi_{l}d\mathbf{x}\\
           \vdots & & \vdots\nonumber\\
           \int\limits_{\Omega} \psi_{l} \psi_{l} \psi_{1}d\mathbf{x} &\ldots & \int\limits_{\Omega} \psi_{l} \psi_{l} \psi_{l}d\mathbf{x}
       \end{bmatrix}
        \mathbf{c}^{e},\mathbf{c}^{e}} 
    \end{bmatrix}, \mathbf{c}^{u, N} }\nonumber\\
    &= \sptb{
    \begin{bmatrix}
         \int\limits_{\Omega} \psi_{1} \psi_{1} \psi_{1}d\mathbf{x} &\ldots&  \int\limits_{\Omega} \psi_{1} \psi_{1} \psi_{l}d\mathbf{x}\nonumber\\
         \vdots & & \vdots\nonumber\\
         \int\limits_{\Omega} \psi_{1} \psi_{l} \psi_{1}d\mathbf{x} &\ldots & \int\limits_{\Omega} \psi_{1} \psi_{l} \psi_{l}d\mathbf{x}
    \end{bmatrix}
    \mathbf{c}^{e},\mathbf{c}^{e}} c_{1}^{\Tilde{u},N} +\ldots +\sptb{
    \begin{bmatrix}
        \int\limits_{\Omega} \psi_{l} \psi_{1} \psi_{1}d\mathbf{x} &\ldots & \int\limits_{\Omega} \psi_{l} \psi_{1} \psi_{l}d\mathbf{x}\nonumber\\
        \vdots & & \vdots\nonumber\\
        \int\limits_{\Omega} \psi_{l} \psi_{l} \psi_{1}d\mathbf{x} &\ldots & \int\limits_{\Omega} \psi_{l} \psi_{l} \psi_{l}d\mathbf{x}
    \end{bmatrix}\mathbf{c}^{e},\mathbf{c}^{e}} c_{l}^{\Tilde{u},N}  \nonumber\\[1.0ex]
    & =\sptb{
    \begin{bmatrix}
        \sum_{q}^{l} c_{q}^{u, N} \int\limits_{\Omega} \psi_{q} \psi_{1} \psi_{1}d\mathbf{x} &\ldots &  \sum_{q}^{l} c_{q}^{u, N} \int\limits_{\Omega} \psi_{q} \psi_{1} \psi_{l}d\mathbf{x}\nonumber\\
        \vdots & & \vdots\nonumber\\
        \sum_{q}^{l} c_{q}^{u, N} \int_{\Omega} \psi_{q} \psi_{l} \psi_{1}d\mathbf{x} &\ldots &  \sum_{q}^{l} c_{q}^{u, N} \int\limits_{\Omega} \psi_{q} \psi_{l} \psi_{l}d\mathbf{x}
    \end{bmatrix}\mathbf{c}^{e},\mathbf{c}^{e}}\nonumber\\
    & =\spta{ \mathbf{T}\left(c^{\Tilde{u}, N}\right)\mathbf{c}^{e},\mathbf{c}^{e} } 
    \label{Eq:quadForm-Tc-ECM}
\end{align}}
Using the idea from the proof of \eqref{InEq:mm_bound-g} in Section 7.7 of \cite{Johnson1987}, it follows that we can use the same argument to bound $\int_{\Omega} \spta{\spta{ \Tilde{e}, \Tilde{e} }, \Tilde{u}^{N}}\; \text{d} \mathbf{x} = \spta{ \mathbf{T}\left(c^{\Tilde{u}, N}\right)\mathbf{c}^{e},\mathbf{c}^{e} }$ \textit{i.e}.,

\begin{lemma}
\label{Lemma:doubleECMKnownSol}
    For all \(\Tilde{e} = \sum_{\tau = 1}^{l}c^{\Tilde{e}}_{\tau} \psi_{\tau},\; \Tilde{u}^{N} = \sum_{\tau = 1}^{l}c^{\Tilde{u}, N}_{\tau} \psi_{\tau}\) and there exists constant $\zeta^{^{e}}_{1}$ and $\zeta^{^{e}}_{_{2}}$ depending on $\eta_{1}$ $\eta_{_{2}}$ defined in \eqref{Eq:h_kBound} and \eqref{Eq:varrho_k-h_kBound}, respectively and independent of the $h$ such that:
     \begin{equation}
          \zeta^{^{e}}_{1} h^{^{2}} \left\|\mathbf{c}^{\Tilde{u}, N}\right\|^{^{2}}_{_{2}}\|\mathbf{c}^{e}\|^{^{2}}_{_{2}}\le \int_{\Omega} \spta{\spta{ \Tilde{e}, \Tilde{e} }, \Tilde{u}^{N}}\; \text{d} \mathbf{x} \le \zeta^{^{e}}_{_{2}} h^{^{2}}\left\|\mathbf{c}^{\Tilde{u}, N} \right\|^{^{2}}_{_{2}}\|\mathbf{c}^{e}\|^{^{2}}_{_{2}}.
           \label{Eq:lemma_tripleproduct}
     \end{equation}
\end{lemma}
\begin{proof}
    Here we focus on an arbitrary element K and transform to a reference element $\Hat{K}$. Let the transformation from a reference element to a physical element be given as the mapping $F:\Hat{K}\mapsto K$, and $x = F(\Hat{x})$. For $\Hat{x}\in P_{1}(\Hat{K})$, $\Hat{e}(\Hat{x}) = \sum_{\tau}c_{\tau}^{\hat{e}}\Hat{\psi}_{\tau}(\Hat{x})$ and $x = (x_{1}, x_{_{2}}) = (h \Hat{x}_{1}, h \Hat{x}_{_{2}}) = h\Hat{x}$. Then the following holds on a single element $K$: 
    \begin{align*}
        \int_{\Omega} \spta{\spta{ \Tilde{e}, \Tilde{e} }, \Tilde{u}^{N}}\; \text{d} \mathbf{x} &= \int_{\Hat{K}}|\sum_{\tau} c^{u,N}_{\tau}\Hat{\psi}_{\tau}||\sum_{\tau} c_{\tau}\Hat{\psi}_{\tau}|^{^{2}}|\text{det}(J)|\text{d}\Hat{\textbf{x}}\nonumber \\
        &=h^{^{2}}\int_{\Hat{K}}|\sum_{\tau} c^{u,N}_{\tau}\Hat{\psi}_{\tau}||\sum_{\tau} c_{\tau}\Hat{\psi}_{\tau}|^{^{2}}\text{d}\Hat{\textbf{x}},
    \end{align*}
    where $|\text{det}(J)| = h^{^{2}}$ is the determinant of the Jacobian of transformation. Note that we only have to sum over all element K to get the whole discretized domain $\Omega$. The rest of the proof follows the same argument from~\cite{Johnson1987}.
\end{proof}
To estimate the condition number $k(\mathbf{A}_{e})$, first we find the quadratic form of each of the terms that constitute $\mathbf{A}_{e}$ (see \eqref{Eq:leftMatrixECM}) and then use their Rayleigh quotients together with \eqref{Eq:quadForm-MassMatrix} and \eqref{Eq:lemma_tripleproduct} to estimate the upper and lower bounds. For the first term, we have that 
\begin{equation}
   \left(\frac{1}{\Delta t} +\alpha_{_{e}}
    \right)\spta{\mathbf{Mc}^{e}, \mathbf{c}^{e}} \le \left(\frac{1}{\Delta t} +\alpha_{_{e}}
    \right)\zeta^{^{g}}_{1}h^{^{2}}\|\mathbf{c}^{e}\|_{_{2}}^{^{2}},
    \label{Ineq:upbnd_massmatrix_ECM}
\end{equation}
and 
\begin{equation}
  \left(\frac{1}{\Delta t} +\alpha_{_{e}}
    \right)\spta{\mathbf{Mc}^{e}, \mathbf{c}^{e}} \ge  \left(\frac{1}{\Delta t} +\alpha_{_{e}}
    \right) \zeta^{^{g}}_{_{2}}h^{^{2}}\|\mathbf{c}^{e}\|_{_{2}}^{^{2}},
    \label{Ineq:lowbnd_massmatrix_ECM}
\end{equation}
 since $\left(\frac{1}{\Delta t} +\alpha_{_{e}} \right)> 0$ and $\mathbf{M}$ is positive definite, where $\zeta^{^{g}}_{1}$ and $\zeta^{^{e}}_{_{2}}$ are from Lemma~\ref{lemma:mass-mat:stiff-mat}. For the second and third term we have
 \begin{align*}
      \mathbf{T}^{e}&:= \left(\alpha_{_{f}} +  2 p_{_{e}}w_{_{e}}\right)\mathbf{T}({c}^{f,N})+ \alpha_{_{m}}\mathbf{T}({c}^{m,N})\nonumber
 \end{align*}
and so
\begin{equation}
      \spta{\mathbf{T}^{e}\mathbf{c}^{e},\mathbf{c}^{e}} \le  \left(\left(\alpha_{_{f}} +  2 p_{_{e}}w_{_{e}}\right)\left\|\mathbf{c}^{f,N}\right\|^{^{2}}_{_{2}} + \alpha_{_{m}}\left\|\mathbf{c}^{m,N}\right\|^{^{2}}_{_{2}}\right) \zeta^{^{e}}_{1}h^{^{2}}\|\mathbf{c}^{e}\|^{^{2}}_{_{2}},
      \label{Ineq:upbnd_tripplematrix_ECM}
\end{equation}
\begin{equation}
     \spta{\mathbf{T}^{e}\mathbf{c}^{e},\mathbf{c}^{e}}\ge \left(\left(\alpha_{_{f}} +  2 p_{_{e}}w_{_{e}}\right)\left\|\mathbf{c}^{f,N}\right\|^{^{2}}_{_{2}} + \alpha_{_{m}}\left\|\mathbf{c}^{m,N}\right\|^{^{2}}_{_{2}}\right) \zeta^{^{e}}_{1}h^{^{2}}\|\mathbf{c}^{e}\|^{^{2}}_{_{2}}.
     \label{Ineq:lowbnd_tripplematrix_ECM}
 \end{equation}
Using \eqref{Ineq:upbnd_massmatrix_ECM} - \eqref{Ineq:lowbnd_tripplematrix_ECM}, we have:
\begin{align*}
    \Lambda_{_{\max}}(\mathbf{A}_{e}) &:=\underset{\mathbf{c}^{e}\neq 0}{\max}\frac{\spta{ \mathbf{M}^{e} \mathbf{c}^{e}, \mathbf{c}^{e} }}{\|\mathbf{c}^{e}\|^{^{2}}} + \underset{\mathbf{c}^{e}\neq 0}{\max}\frac{\spta{ \mathbf{T}^{e} \mathbf{c}^{e}, \mathbf{c}^{e} }}{\|\mathbf{c}^{e}\|^{^{2}}}\nonumber\\
    &\le  \left(\frac{1}{\Delta t} +\alpha_{_{e}}
    \right)\zeta^{^{g}}_{1} + \left(\left(\alpha_{_{f}} +  2 p_{_{e}}w_{_{e}}\right)\left\|\mathbf{c}^{f,N}\right\|^{^{2}}_{_{2}} + \alpha_{_{m}}\left\|\mathbf{c}^{m,N}\right\|^{^{2}}_{_{2}}\right)\zeta^{^{e}}_{1}h^{^{2}}
\end{align*}
and 
\begin{align*}
    \Lambda_{_{\text{min}}}(\mathbf{A}_{e}) &=\underset{\mathbf{c}\neq 0}{\text{min}}\frac{\spta{ \mathbf{M}^{^{e}} \mathbf{c}^{e}, \mathbf{c}^{e} }}{\|\mathbf{c}\|^{^{2}}} + \underset{\mathbf{c}\neq 0}{\text{min}}\frac{\spta{ \mathbf{T}^{^{e}}\mathbf{c}, \mathbf{c} }}{\|\mathbf{c}\|^{^{2}}}\nonumber\\
    &\ge  \left(\frac{1}{\Delta t} +\alpha_{_{e}}
    \right)\zeta^{^{g}}_{1} + \left(\left(\alpha_{_{f}} +  2 p_{_{e}}w_{_{e}}\right)\left\|\mathbf{c}^{f,N}\right\|^{^{2}}_{_{2}} + \alpha_{_{m}}\left\|\mathbf{c}^{m,N}\right\|^{^{2}}_{_{2}}\right)\zeta^{^{e}}_{1}h^{^{2}},
\end{align*}
and the condition number is given by
\begin{equation}
  k(\mathbf{A}_{e}) = \frac{\Lambda_{_{\max}}(\mathbf{A}_{e}) }{\Lambda_{_{\text{min}}}(\mathbf{A}_{e}) } \le\frac{\left(\frac{1}{\Delta t} +\alpha_{_{e}}
    \right)\zeta^{^{g}}_{1} + \left(\left(\alpha_{_{f}} +  2 p_{_{e}}w_{_{e}}\right)\left\|\mathbf{c}^{f,N}\right\|^{^{2}}_{_{2}} + \alpha_{_{m}}\left\|\mathbf{c}^{m,N}\right\|^{^{2}}_{_{2}}\right)\zeta^{^{e}}_{1}h^{^{2}}}{\left(\frac{1}{\Delta t} +\alpha_{_{e}}
    \right)\zeta^{^{g}}_{1} + \left(\left(\alpha_{_{f}} +  2 p_{_{e}}w_{_{e}}\right)\left\|\mathbf{c}^{f,N}\right\|^{^{2}}_{_{2}} + \alpha_{_{m}}\left\|\mathbf{c}^{m,N}\right\|^{^{2}}_{_{2}}\right)\zeta^{^{e}}_{1}h^{^{2}}}\le C.
\end{equation}

\paragraph{Condition number for the $f$-equation.}
For the corresponding fibroblast dynamic, we get the following after discretizing in time and integrating by parts
\begin{align}
    \int_{\Omega}\Bigg[\frac{f - f^{N}}{\Delta t} v_{_{f}} + D_{_{f}}\spta{ \nabla f, \nabla v_{_{f}}}   -\mu_{_{f}}\spta{ A_{f}^{N}[u],f \,\nabla  v_{_{f}}}  
    + \lambda_{_{f}} f v_{_{f}} \nonumber\\
    - p_{_{f}} g^{N} f\left(1-w_{_{g}} g^{N} - w_{_{f}} f - w_{_{m}} m^{N} - w_{_{e}} e^{N}\right) v_{_{f}}\Bigg]\text{d}\mathbf{x} = 0,
    \label{Eq:fib:timeDisc}        
\end{align}
where 
\begin{align*}
    A_{_{f}}^{N}[u](\mathbf{x}) = \frac{1}{R}{\int\limits_{\textbf{B}_{_{\|\cdot\|_{_{\infty}}}}\!\!(\textbf{0},R)}}\!\!\!\mathrm{K}(\|\mathbf{y}\|_{_{2}})  \mathbf{n}(\mathbf{y})\left(\!\!\left(1-\rho\left(u^{N}\right)\!\!\right)^{+}\,{\Gamma}_{_i}\left(u^{N}\right)\!\!\right)(\mathbf{x}+\mathbf{y},t)\, \text{d}\mathbf{y}.
\end{align*}
We see that the fifth term on the left-hand-side of equation \eqref{Eq:fib:timeDisc} (term which we denote as $T^{f}_{_{5}}$) is non-linear in the unknown $f$. For that reason, we linearize it with respect to the unknown values $f$ evaluated at the known (\textit{i.e}., at time $N$) values of the other components $g, m,$ and $e$: 
\begin{equation}
    \frac{d T_{_{5}}}{d f} = 2 p_{_{f}}g^{N}w_{_{f}} f - p_{_{f}}g^{N}\left(1 \!-\! w_{_{g}}g^{N}\!-\! w_{_{m}} m^{N}\!-\! w_{_{e}} e^{N}\right).
    \label{Eq:linearizing_logistic_f}
\end{equation}
Furthermore, we approximate some terms in the integrand as follows:
\begin{equation*}
    0\le \left(1-\rho\left(u^{N}(\mathbf{x}+\mathbf{y},t)\right)\right)^{^{\!\!+}} = \left(\!1-\!(w_{_{g}} g^{N}\!\!\! +\! w_{_{f}} f^{N}\!\!\! + \!w_{_{m}} m^{N}\!\!\! +\!w_{_{e}} e^{N})(\mathbf{x}+\mathbf{y},t)\!\right)^{^{\!\!+}} \le 1,
\end{equation*} 
and 
\begin{align*}
    {\Gamma}_{_f}\left(u^{N}\right)(\mathbf{x}+\mathbf{y},t) &= \left(\!\!\frac{e^{N}\!\! + \!g^{N}}{1\!+\!e^{N}\!\! + \!g^{N}}\!\!\right)\!\!\!\left(\!S_{_{ff}}^{\max} f^{N} \!\!+S_{_{fm}}^{\max} m^{N} \!\!+S_{_{fe}}^{\max} e^{N}\!\right)(\mathbf{x}+\mathbf{y},t)\\
    &\le S_{_{ff}}^{\max} f^{N}(\mathbf{x}+\mathbf{y},t) \!\!+S_{_{fm}}^{\max} m^{N}(\mathbf{x}+\mathbf{y},t) \!\!+S_{_{fe}}^{\max} e^{N}(\mathbf{x}+\mathbf{y},t),
\end{align*}
where we have used the fact that $\left(e^{N}\!\! + \!g^{{N}}\right)/\left(1\!+\!e^{{N}}\!\! + \!g^{{N}}\right)\le 1$. Approximating 
each $ j^{N}(\mathbf{x}+\mathbf{y},t), j\in\{f, m, e\}$ with its Taylor expansion and assuming higher order terms vanish we have
\begin{align*}
   j^{N}(\mathbf{x}+\mathbf{y},t) \approx j^{N}\!(\mathbf{x},t) + \spta{\! \nabla j^{N}\!(\mathbf{x},t), \,\mathbf{y}\!},\;\;\;j\in\{f, m, e\}.
\end{align*}
Substituting the unit normal defined in \eqref{surf:normal:def} and the kernel defined in \eqref{kernel:Gaussian} together with the above approximations, the non-local term $A_{_{f}}^{N}[u](\mathbf{x})$ can now be approximated by $\Tilde{A}_{_{f}}^{N}[u](\mathbf{x})$ that is given by
{\small
\begin{align}
    \Tilde{A}_{_{f}}^{N}[u](\mathbf{x}) =&  \frac{1}{R}\!\!{\int\limits_{\textbf{B}_{_{\|\cdot\|_{_{\infty}}}}\!\!(\textbf{0},R)}}\!\! \!\!\!\! \frac{\mathbf{y} e^{^{-{\|\mathbf{y}\|_{_{2}}^{^{2}}\big/2\sigma^{^{2}}}}} }{2\pi\sigma^{^{2}}}\,\!\!\!\Big[\left(\!S_{_{ff}}^{^{\max}}\left(f^{N}\!(\mathbf{x},t)\right. + \spt{\! \nabla f^{N}\!(\mathbf{x},t), \,\mathbf{y}\!}\right)+
    \left(\!S_{_{fm}}^{\max}\left(m^{N}\!(\mathbf{x},t) + \spt{\! \nabla m^{N}\!(\mathbf{x},t), \,\mathbf{y}\!}\right)\right.\nonumber\\
    &+\left(S_{_{fe}}^{\max} \left(e^{N}\!(\mathbf{x},t) + \spt{\!\nabla e^{N}\!(\mathbf{x},t), \,\mathbf{y}\!}\right)\!\right)\Big] \text{d}\mathbf{y} \nonumber\\
    =& \frac{1}{2\pi\sigma^{^{2}}R}\Bigg[(\!S_{_{ff}}^{^{\max}}f^{N}\!(\mathbf{x},t) + S_{_{fm}}^{^{\max}}m^{N}\!(\mathbf{x},t) + S_{_{fe}}^{^{\max}} e^{N}\!(\mathbf{x},t)) \!\!\!\!\!\!\int\limits_{\textbf{B}_{_{\|\cdot\|_{_{\infty}}}}\!\!(\textbf{0},R)} \!\!\!\!\!\!\mathbf{y} e^{^{-{\|\mathbf{y}\|_{_{2}}^{^{2}}\big/2\sigma^{^{2}}}}}\text{d}\mathbf{y}\nonumber\\
    & + \Bigg[S_{_{ff}}^{^{\max}}\spt{ \nabla f^{N}, \!\!\!\!\!\!\int\limits_{\textbf{B}_{_{\|\cdot\|_{_{\infty}}}}\!\!(\textbf{0},R)}\!\!\!\!\!\!e^{^{-{\|\mathbf{y}\|_{_{2}}^{^{2}}\big/2\sigma^{^{2}}}}}
    \begin{bmatrix}
        y_{1}^{2} \\
        y_{1} y_{2}
    \end{bmatrix}
    \text{d}\mathbf{y}} 
    +   S_{_{fm}}^{^{\max}}\spt{ \nabla m^{N}, \!\!\!\!\!\!\int\limits_{\textbf{B}_{_{\|\cdot\|_{_{\infty}}}}\!\!(\textbf{0},R)}\!\!\!\!\!\!e^{^{-{\|\mathbf{y}\|_{_{2}}^{^{2}}\big/2\sigma^{^{2}}}}}
    \begin{bmatrix}
        y_{1}^{2} \\
        y_{1} y_{2}
    \end{bmatrix}\text{d}\mathbf{y}} \nonumber\\
     &+   S_{_{fe}}^{^{\max}}\spt{ \nabla e^{N}, \!\!\!\!\!\!\int\limits_{\textbf{B}_{_{\|\cdot\|_{_{\infty}}}}\!\!(\textbf{0},R)}\!\!\!\!\!\!e^{^{-{\|\mathbf{y}\|_{_{2}}^{^{2}}\big/2\sigma^{^{2}}}}}
     \begin{bmatrix}
        y_{1}^{2} \\
        y_{1} y_{2}
    \end{bmatrix}\text{d}\mathbf{y}},\,\, 
     S_{_{ff}}^{^{\max}}\spt{ \nabla f^{N}, \!\!\!\!\!\!\int\limits_{\textbf{B}_{_{\|\cdot\|_{_{\infty}}}}\!\!(\textbf{0},R)}\!\!\!\!\!\!e^{^{-{\|\mathbf{y}\|_{_{2}}^{^{2}}\big/2\sigma^{^{2}}}}}
    \begin{bmatrix}
        y_{1} y_{2} \\
        y_{2}^{2}
    \end{bmatrix}
\text{d}\mathbf{y}} \nonumber\\
    &+  S_{_{fm}}^{^{\max}}\spt{ \nabla m^{N}, \!\!\!\!\!\!\int\limits_{\textbf{B}_{_{\|\cdot\|_{_{\infty}}}}\!\!(\textbf{0},R)}\!\!\!\!\!\!e^{^{-{\|\mathbf{y}\|_{_{2}}^{^{2}}\big/2\sigma^{^{2}}}}}
    \begin{bmatrix}
        y_{1} y_{2} \\
        y_{2}^{2}
    \end{bmatrix}\text{d}\mathbf{y}} 
     +   S_{_{fe}}^{^{\max}}\spt{ \nabla e^{N}, \!\!\!\!\!\!\int\limits_{\textbf{B}_{_{\|\cdot\|_{_{\infty}}}}\!\!(\textbf{0},R)}\!\!\!\!\!\!e^{^{-{\|\mathbf{y}\|_{_{2}}^{^{2}}\big/2\sigma^{^{2}}}}}
     \begin{bmatrix}
        y_{1} y_{2} \\
        y_{2}^{2}
    \end{bmatrix}\text{d}\mathbf{y}}
     \Bigg]
     \Bigg]\nonumber\\
     =& \frac{1}{2\pi\sigma^{^{2}}R}\Bigg[S_{_{ff}}^{^{\max}}\spt{ \nabla f^{N}, 
     \begin{bmatrix}
        C_{_{\rho,R}} \\
        0
    \end{bmatrix}} + S_{_{fm}}^{^{\max}}\spt{ \nabla m^{N}, 
    \begin{bmatrix}
        C_{_{\rho,R}} \\
        0
    \end{bmatrix}} S_{_{fe}}^{^{\max}}\spt{ \nabla e^{N}, 
    \begin{bmatrix}
        C_{_{\rho,R}} \\
        0
    \end{bmatrix}} , \nonumber\\
      &S_{_{ff}}^{^{\max}}\spt{ \nabla f^{N}, 
    \begin{bmatrix}
        0 \\
        C_{_{\rho,R}}
    \end{bmatrix}} + S_{_{fm}}^{^{\max}}\spt{ \nabla m^{N}, 
    \begin{bmatrix}
        0 \\
        C_{_{\rho,R}}
    \end{bmatrix}} S_{_{fe}}^{^{\max}}\spt{ \nabla e^{N}, 
    \begin{bmatrix}
        0 \\
        C_{_{\rho,R}}
    \end{bmatrix}}\Bigg] \nonumber\\
      =& \frac{C_{_{\rho, R}}}{2\pi \rho^{^{2}}R}\left(\!S_{_{ff}}^{^{\max}} \nabla f^{N}\!(\mathbf{x},t) + S_{_{fm}}^{^{\max}}\nabla m^{N}\!(\mathbf{x},t)+S_{_{fe}}^{^{\max}} \nabla e^{N}\!(\mathbf{x},t)\right),
       \label{Eq:nonlocal:Approx-f}
\end{align}}
where 
\[ 
C_{_{\sigma, R}} = \frac{-\sigma\,\text{ erf} \left(\frac{R\sqrt{2}}{2\sigma}\right)\left[\sqrt{2\pi} \,\,e^{-R^{^{2}}/2\sigma^{^{2}}} R - \pi\sigma \text{ erf}\left(\frac{R\sqrt{2}}{2\sigma}\right)\right]}{\pi},
\] 
since \[\int\limits_{\textbf{B}_{_{\|\cdot\|_{_{\infty}}}}\!\!(\textbf{0},R)}\!\!\!\!\!\!e^{^{-{\|\mathbf{y}\|_{_{2}}^{^{2}}\big/2\sigma^{^{2}}}}} y_{i} y_{j}\text{d}\mathbf{x} =
\begin{cases} C_{_{\sigma, R}}, & i=j\\
    0, &i\neq j
\end{cases}  
\quad i,j\in\{1,2\}
\].

Now that we are done with the various approximations, we substitute the approximation of \(T^{f}_{5}\) in~\eqref{Eq:linearizing_logistic_f} and the approximation of the non-local term $ A_{_{f}}^{N}[u](\mathbf{x})$ in~\eqref{Eq:nonlocal:Approx-f} into equation \eqref{Eq:fib:timeDisc} and integrate by parts to obtain

\begin{align}
    \int_{\Omega}\Bigg[\frac{f - f^{N}}{\Delta t} v_{_{f}} + D_{_{f}}\spta{ \nabla f, \nabla v_{_{f}}} 
     -\mu_{_{f}} \spta{ \Tilde{A}_{f}^{N}[u], f \nabla  v_{_{f}}} +\lambda_{_{f}} f v_{_{f}} +  2 p_{_{f}}g^{N}w_{_{f}} f\nonumber\\
     - p_{_{f}}g^{N}\left(1 \!-\! w_{_{g}}g^{N}\!-\! w_{_{m}} m^{N}\!-\! w_{_{e}} e^{N}\right) v_{_{f}}\Bigg]\text{d}\mathbf{x} = 0, \quad\forall v_{_{f}}.
     \label{Eq:weakForm-fib}
\end{align}
This further becomes
\begin{align*}
   &  \int_{\Omega}\Bigg[\frac{f - f^{N}}{\Delta t} v_{_{f}} + D_{_{f}}\spta{ \nabla f, \nabla v_{_{f}}} 
     -\frac{\mu_{_{f}} C_{_{\rho, R}}}{R}\left(S_{_{ff}}^{^{\max}}f  \spta{ \nabla f^{N},\nabla  v_{_{f}}  } \right.+
     S_{_{fm}}^{^{\max}}f  \spta{ \nabla m^{N},\nabla  v_{_{f}}  } \\
     &+S_{_{fe}}^{^{\max}} f  \left.\spta{ \nabla e^{N},\nabla  v_{_{f}}  } \right)   
      +\lambda_{_{f}} f v_{_{f}} +  2 p_{_{f}}g^{N}w_{_{f}} f - p_{_{f}}g^{N}\left(1 \!-\! w_{_{g}}g^{N}\!-\! w_{_{m}} m^{N}\!-\! w_{_{e}} e^{N}\right) v_{_{f}}\Bigg]\text{d}\mathbf{x} = 0, \quad\forall v_{_{f}},
\end{align*}
where we have used the boundary condition in~\eqref{EqBC:ExisUniq}, which reduces simply to \(\spta{ \nabla f, n}= \spta{ \nabla m, n}=0\) upon the approximation of the non-local term~\(A_{_{f}}^{N}[u](\mathbf{x})\) in \eqref{Eq:nonlocal:Approx-f}. Furthermore, upon discretization in space we have
\begin{align*}
    &\sum_{\tau}^{l}\Bigg[\left(\frac{1}{\Delta t}+\lambda_{_{f}} \right)\int_{\Omega}\psi_{\tau}\psi_{o}\,\text{d}\mathbf{x} + \sum_{q}^{l}2p_{_{f}} w_{_{f}}c_{q}^{g,N}\int_{\Omega}\psi_{q}\psi_{\tau}\psi_{o}\,\text{d}\mathbf{x} + D_{_{f}} \int_{\Omega} \spta{ \nabla\psi_{\tau}, \nabla\psi_{o}} \text{d}\mathbf{x}   \,\,-\nonumber\\
    &\frac{\mu_{_{f}} C_{_{\rho, R}}}{R}\sum_{q}\Bigg[\left(S_{_{ff}}^{^{\max}} c_{q}^{f,N} +S_{_{fm}}^{^{\max}} c_{q}^{m,N}+S_{_{fe}}^{^{\max}} c_{q}^{e,N}\right)\int_{\Omega}\spta{ \nabla\psi_{q}, \nabla \psi_{o}} \psi_{\tau}\:\text{d}\mathbf{x}\Bigg]\Bigg]{c}_{\tau}^{f} \\ 
   & =\int\limits_{\Omega}\sum_{\tau = 1}^{l} \Bigg[ p_{_{f}} c_{\tau}^{g, N}\psi_{\tau} \left(1 - \sum_{s = 1}^{l}\psi_{s} \left(c_{s}^{g,N},c_{s}^{m,N},c_{s}^{e,N} \right)\right) 
    +\frac{1}{\Delta t} c_{\tau}^{e,N}\psi_{\tau}\Bigg]\psi_{o}\;\text{d}\textbf{x}\quad\forall o\in\{1,\dots, l\}.
\end{align*}
We observe that this is already in the general form
\begin{equation*}
    \mathbf{A}_{f} \mathbf{c}^{f}= \mathbf{F}_{_f},
\end{equation*}
where the entries of $\mathbf{A}_{f}\in\mathbb{R}^{l\times l}$ are given as
\begin{align*}
    (\mathbf{A}_{f})_{\tau o} = &\left(\frac{1}{\Delta t}+\lambda_{_{f}} \right)\int_{\Omega}\psi_{\tau}\psi_{o}\,\text{d}\mathbf{x} + \sum_{q}^{l}2p_{_{f}} w_{_{f}}c_{q}^{g,N}\int_{\Omega}\psi_{q}\psi_{\tau}\psi_{o}\,\text{d}\mathbf{x} + D_{_{f}} \int_{\Omega} \spta{ \nabla\psi_{\tau}, \nabla\psi_{o}} \text{d}\mathbf{x}   \,\,-\nonumber\\
    &\frac{\mu_{_{f}} C_{_{\rho, R}}}{R}\sum_{q}\Bigg[\left(S_{_{ff}}^{^{\max}} c_{q}^{f,N} +S_{_{fm}}^{^{\max}} c_{q}^{m,N}+S_{_{fe}}^{^{\max}} c_{q}^{e,N}\right) \int_{\Omega}\spta{ \nabla\psi_{q}, \nabla \psi_{o}} \psi_{\tau}\;\text{d}\mathbf{x}\Bigg],
\end{align*}
and the entries of $(\mathbf{F}_{f})_{o}\in\mathbb{R}^{l}$ as
\begin{equation*}
    (\mathbf{F}_{f})_{o} = \int\limits_{\Omega}\sum_{\tau = 1}^{l} \Bigg[ p_{_{f}} c_{\tau}^{g, N}\psi_{\tau} \left(1 - \sum_{s = 1}^{l}\psi_{s} \left(c_{s}^{g,N},c_{s}^{m,N},c_{s}^{e,N} \right)\right)
    +\frac{1}{\Delta t} c_{\tau}^{e,N}\psi_{\tau}\Bigg]\psi_{o}\;\text{d}\textbf{x}.
\end{equation*}
We note that we can write \(\mathbf{A}_{e}\) in the compact form
\begin{align}
\label{Eq:finalLHSmatrixFIB}
   \mathbf{A}_{f} = \left(\frac{1}{\Delta t}+\lambda_{_{f}}\right) \mathbf{M} + 2p_{_{f}} w_{_{f}}\mathbf{T}({c}^{g, N}) + D_{_{f}} \mathbf{K} - 
    \frac{\mu_{_{f}} C_{_{\rho, R}}}{R}\left(S_{_{ff}}^{^{\max}} \mathbf{H}({c}^{f,N})
    +S_{_{fm}}^{^{\max}} \mathbf{H}({c}^{m,N})+S_{_{fe}}^{^{\max}}\mathbf{H}({c}^{e,N})\right),
\end{align}
where $\mathbf{M}$ is the usual mass matrix, 
\begin{align*}
\mathbf{T}\left(c^{g, N}\right) =  \begin{bmatrix}
        \sum_{q}^{l} c_{q}^{g, N} \int\limits_{\Omega} \psi_{q} \psi_{1} \psi_{1}d\mathbf{x} &\ldots &  \sum_{q}^{l} c_{q}^{g, N} \int\limits_{\Omega} \psi_{q} \psi_{1} \psi_{l}d\mathbf{x}\\
        \vdots & & \vdots\\
          \sum_{q}^{l} c_{q}^{g, N} \int_{\Omega} \psi_{q} \psi_{l} \psi_{1}d\mathbf{x} &\ldots &  \sum_{q}^{l} c_{q}^{g, N} \int\limits_{\Omega} \psi_{q} \psi_{l} \psi_{l}d\mathbf{x}\end{bmatrix},
\end{align*} 
and 
\begin{align*}
    \mathbf{H}({c}^{i,N}) =  
    \begin{bmatrix}
        \sum_{q=1}^{l}c^{i,N}_{q}\int\limits_{\Omega} \spta{ \nabla\psi_{q}, \nabla\psi_{1}} \psi_{1}d\mathbf{x} &\ldots & \sum_{q=1}^{l}c^{i,N}_{q}\int\limits_{\Omega} \spta{ \nabla\psi_{q}, \nabla\psi_{l}} \psi_{1}d\mathbf{x}\\
        \vdots & & \vdots\\
        \sum_{q=1}^{l}c^{i,N}_{q}\int_{\Omega} \spta{ \nabla\psi_{q}, \nabla\psi_{1}} \psi_{l}d\mathbf{x} &\ldots &  \sum_{q=1}^{l}c^{i,N}_{q}\int\limits_{\Omega} \spta{ \nabla\psi_{q}, \nabla\psi_{l}} \psi_{l}d\mathbf{x}
    \end{bmatrix}, \quad i\in\{f, m, e\}.
\end{align*} 
To estimate the condition number of \(\mathbf{A}_{f}\in\mathbb{R}^{l\times l}\), just as for the \(g\)-equation we will first derive its quadratic form \(\spta{ \mathbf{A}_{f} \mathbf{c}^{f},\mathbf{c}^{f}}\) by deriving the quadratic form of each of the matrices that constitutes it (since \(\mathbf{A}_{f}\) is simply a linear combination of theses matrices). Note that we only do this for each of the matrices $\mathbf{H}({c}^{i,N}), i\in\{f,m,e\}$ , since the quadratic forms of the matrices \(\mathbf{M}, \mathbf{K}\) and \(\mathbf{T}(c^{u,N}), u\in\{g, f, m, e\}\), have been previous derived in \eqref{Eq:quadForm-MassMatrix}, \eqref{Eq:quadForm-StiffMatrix} and \eqref{Eq:quadForm-Tc-ECM}, respectively.
\\

For an arbitrary \(\Tilde{f} = \sum_{\tau =1}^{l} c_{\tau}^{f} \psi_{\tau}\) and \(\Tilde{u}^{N} = \sum_{\tau =1}^{l} c^{u,N}_{\tau} \psi_{\tau}\), where \(\Tilde{u}\) could be any of the variables \(\{g,f,m,e\}\), we have that
{\small
\begin{align}
    \int\limits_{\Omega}\spta{ \nabla \Tilde{f} , \nabla \Tilde{f}}\, \Tilde{u}^{N}\:\text{d}\mathbf{x} =\:&
     \sptb{
    \begin{bmatrix}
        \sptb{
        \begin{bmatrix}
            \int\limits_{\Omega} \spt{\nabla \psi_{1}, \nabla\psi_{1}} \psi_{1}d\mathbf{x} &\ldots&  \int\limits_{\Omega} \spta{\nabla\psi_{1} , \nabla\psi_{1}} \psi_{l}d\mathbf{x}\\
            \vdots & & \vdots\\
            \int\limits_{\Omega} \spt{\nabla\psi_{1}, \nabla\psi_{l}} \psi_{1}d\mathbf{x} &\ldots & \int\limits_{\Omega} \spta{\nabla\psi_{1}, \nabla\psi_{l}} \psi_{l}d\mathbf{x}
        \end{bmatrix}
        \mathbf{c}^{f},\mathbf{c}^{f}} \nonumber\\
        \vdots\\
        \sptb{
        \begin{bmatrix}
            \int\limits_{\Omega} \spt{\nabla\psi_{l}, \nabla\psi_{1}} \psi_{1}d\mathbf{x} &\ldots&  \int\limits_{\Omega} \spta{\nabla\psi_{l}, \nabla\psi_{1}} \psi_{l}d\mathbf{x}\\
            \vdots & & \vdots\\
            \int\limits_{\Omega} \spt{\nabla\psi_{l}, \nabla \psi_{l}} \psi_{1}d\mathbf{x} &\ldots & \int\limits_{\Omega} \spta{\nabla\psi_{l}, \psi_{l} }\psi_{l}d\mathbf{x}
        \end{bmatrix}
        \mathbf{c}^{f},\mathbf{c}^{f}} 
    \end{bmatrix}, \mathbf{c}^{u, N} } \nonumber\\
    =& \sptb{
    \begin{bmatrix}
        \int\limits_{\Omega} \spt{\nabla \psi_{1}, \nabla\psi_{1}} \psi_{1}d\mathbf{x} &\ldots&  \int\limits_{\Omega} \spt{\nabla\psi_{1} , \nabla\psi_{l}} \psi_{1}d\mathbf{x}\\
        \vdots & & \vdots\\
        \int\limits_{\Omega} \spt{\nabla\psi_{1}, \nabla\psi_{1}} \psi_{l}d\mathbf{x} &\ldots & \int\limits_{\Omega} \spt{\nabla\psi_{1}, \nabla\psi_{l}} \psi_{l}d\mathbf{x}
    \end{bmatrix}
    \mathbf{c}^{f},\mathbf{c}^{f}} c_{1}^{u,N} \nonumber \\ &+\ldots+\sptb{
    \begin{bmatrix}
        \int\limits_{\Omega} \spt{\nabla\psi_{l}, \nabla\psi_{1}} \psi_{1}d\mathbf{x} &\ldots&  \int\limits_{\Omega} \spta{\nabla\psi_{l}, \nabla\psi_{l}} \psi_{1}d\mathbf{x}\\
        \vdots & & \vdots\\
        \int\limits_{\Omega} \spta{\nabla\psi_{l}, \nabla \psi_{1}} \psi_{l}d\mathbf{x} &\ldots & \int\limits_{\Omega} \spt{\nabla\psi_{l}, \psi_{l} }\psi_{l}d\mathbf{x}
    \end{bmatrix}\mathbf{c}^{f},\mathbf{c}^{f}} c_{l}^{u,N}   \nonumber\\[1.0ex]
    =&\sptb{
    \begin{bmatrix}
        \sum_{q}^{l} c_{q}^{u, N} \int\limits_{\Omega} \spt{\nabla\psi_{q},\nabla \psi_{1} }\psi_{1}d\mathbf{x} &\ldots &  \sum_{q}^{l} c_{q}^{u, N} \int\limits_{\Omega} \spta{\nabla\psi_{q} ,\nabla\psi_{l} } \psi_{1}d\mathbf{x}\\
        \vdots & & \vdots\\
        \sum_{q}^{l} c_{q}^{u, N} \int\limits_{\Omega} \spta{\nabla\psi_{q}, \nabla \psi_{1} }\psi_{l}d\mathbf{x} &\ldots &  \sum_{q}^{l} c_{q}^{u, N} \int\limits_{\Omega} \spta{\nabla\psi_{q}, \nabla \psi_{l} }\psi_{l}d\mathbf{x}
    \end{bmatrix}\mathbf{c}^{f},\mathbf{c}^{f}} \nonumber\\[1.0ex]
    =&\spta{ \mathbf{H}({c}^{u,N})\mathbf{c}^{f},\mathbf{c}^{f} }. 
   \label{Eq:quadFormMatrix-StiffMatUn}
\end{align}}
\begin{remark}
\label{rem:Af}
  Note that, preserving the precise implicit scheme that we implemented for our previous numerical results in refs.~\cite{AdebayoEtAl2023,AdebayoTrucuEftimie2025}, which considered the approximation for the linearization of the adhesion flux \(-\mu_{_{f}}\int_{\Omega}\spta{A_{f}^{N}[u],f^{N+1}\,\nabla  v_{_{f}}} \text{d}\mathbf{x} \), leads to 
   a non-symmetric matrix emerging from the discretization of the term \(-\mu_{_{f}}\int_{\Omega}\spta{\Tilde{A}^{N}_{f}[u],f^{N+1}\nabla v_{f}}  \text{d}\mathbf{x}\)  that appears in equation \eqref{Eq:weakForm-fib}, which complicates significantly the analysis results for the estimation of the condition number. To address this issue, we propose an amendment of the initial numerical scheme by which the approximation of the adhesion flux is rather given by \(-\mu_{_{f}}\int_{\Omega}\spta{ \Tilde{A}^{N}_{f}[u], f^{N}\nabla v_{f}} \). In support of this amendment, we carried out a series of numerical tests in which the plots of the $L^{2}$ norm of the difference in the results (obtained with the two approximations of the adhesion flux) are shown in Figure~\ref{Fig-L2NormDifference}. We observe in this figure that the $L^{2}$ norm of the difference is very small (i.e., it remains below \(10^{-8}\) for the fibroblasts, and below $10^{-4}$ for the rest of the variables), which strengthens the case for us to continue the condition number analysis with the amended scheme.
\end{remark}

\begin{figure}
    \centering \includegraphics[width=4.5in]{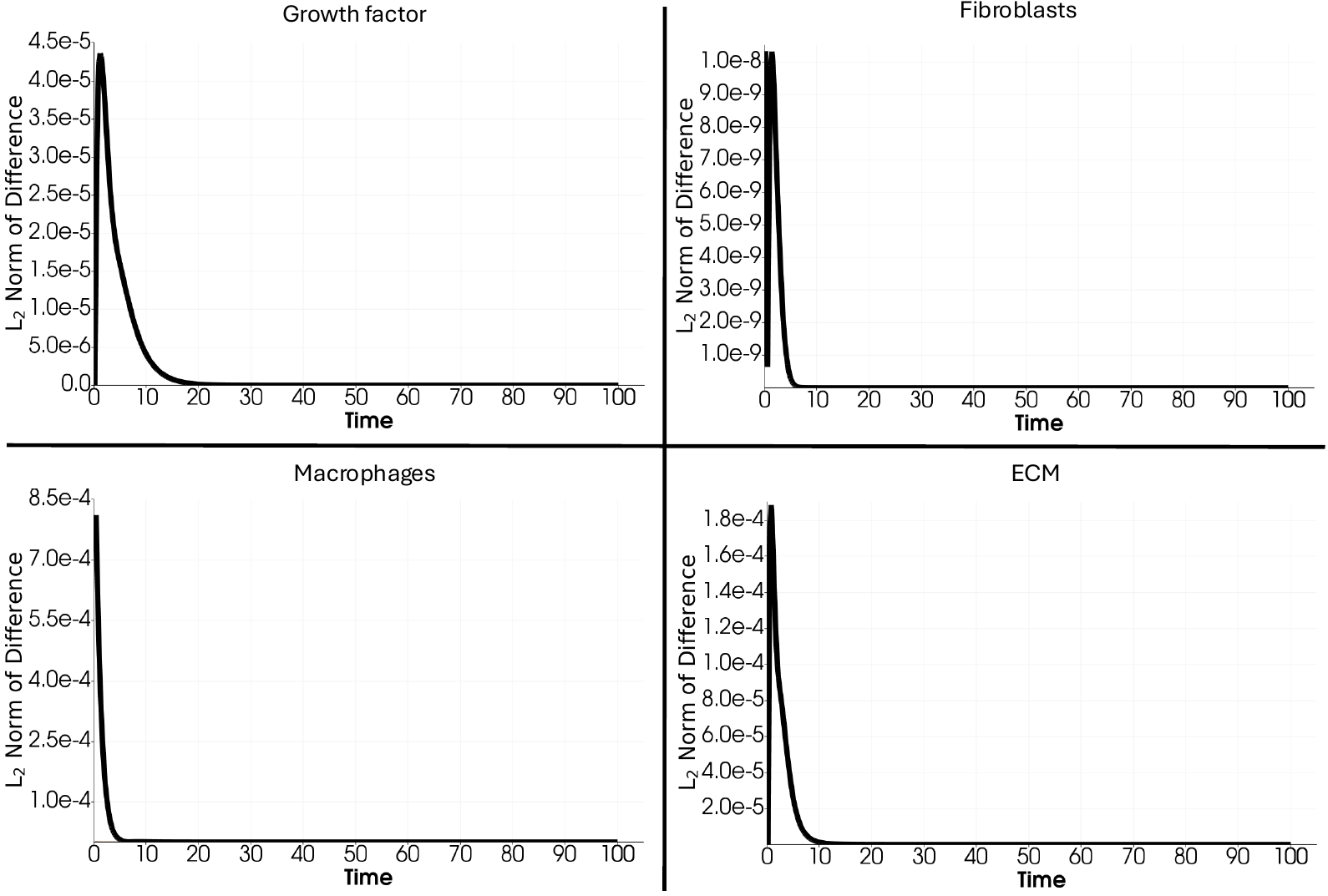}
    \caption{\small
            Numerical simulation showing the L$_{_{2}}$ norm of the differences between when the system is solved strictly using the initial numerical scheme where the approximation of the adhesion flux is given by \(-\mu_{_{f}}\spta{ \Tilde{A}_{f}^{N}[u],f^{N+1} \,\nabla  v_{_{f}}} \) and when it is solved using the amended scheme where the approximation of the adhesion is now given by \(-\mu_{_{f}}\spta{ \Tilde{A}_{f}^{N}[u],f^{N} \,\nabla  v_{_{f}}} \). We show this difference for the growth factor, fibroblasts, macrophages and ECM  as time goes from $t = 0$ to $t = 100$. The initial conditions and numerical setup are described in Appendices~\ref{appendix-full-discretization} and \ref{Appendix:num-det}. The parameter values for the simulations are shown in Table \ref{Table-Param}.}
            \label{Fig-L2NormDifference}
\end{figure}
\noindent Based on Remark \ref{rem:Af}, \(\mathbf{A}_{f}\) now reduces to 
\begin{align}
    \mathbf{A}_{f} = \left(\frac{1}{\Delta t}+\lambda_{_{f}}\right) \mathbf{M} + 2p_{_{f}} w_{_{f}}\mathbf{T}({c}^{g, N}) + D_{_{f}} \mathbf{K},
    \label{Eq:final-A_f}
\end{align}
and \(\mathbf{F}_{f}\) now has entries
\begin{align*}
     (\mathbf{F}_{f})_{o} = \int\limits_{\Omega}\sum_{\tau = 1}^{l} \Bigg[ p_{_{f}} c_{\tau}^{g, N}\psi_{\tau} \left(1 - \sum_{s = 1}^{l}\psi_{s} \left(c_{s}^{g,N},c_{s}^{m,N},c_{s}^{e,N} \right)\right)
    +\frac{1}{\Delta t} c_{\tau}^{e,N}\psi_{\tau}\Bigg]\psi_{o}\;\text{d}\textbf{x} - \\
    \frac{\mu_{_{f}} C_{_{\rho, R}}}{R}\sum_{q}\Bigg[\left(S_{_{ff}}^{^{\max}} c_{q}^{f,N} +S_{_{fm}}^{^{\max}} c_{q}^{m,N}+S_{_{fe}}^{^{\max}} c_{q}^{e,N}\right)\int_{\Omega}\spta{ \nabla\psi_{q}, \nabla \psi_{o}} \psi_{\tau}\:\text{d}\mathbf{x}\Bigg].
\end{align*}
Note that the approximation of the adhesion flux now contains only known terms (\textit{i.e}., terms at time $t= N$) in the amended numerical scheme, and therefore they are moved to the right-hand-side matrix \(\mathbf{F}_{f}\). 
Finally, to obtain the Rayleigh quotient of \(\mathbf{A}_{f}\), we derive bounds for the quadratic forms each matrix that makes  up \(\mathbf{A}_{f}\) (see eq.~\eqref{Eq:final-A_f}). To this end, the following remark is an extension of Lemma~\ref{lemma:mass-mat:stiff-mat} and a completion of Lemma~\ref{Lemma:doubleECMKnownSol}.
\begin{remark}
\label{Rem:Fib}
    Given that Lemma~\ref{lemma:mass-mat:stiff-mat} holds, then we have that for any arbitrary $\Tilde{f}= \sum_{\tau}^{l}c_{\tau}^{f}\psi_{\tau}$ and $\Tilde{u}^{N}= \sum_{\tau}^{l}c_{\tau}^{u, N}\psi_{\tau},\: u\in~\{g,f,m,e\}$, there exist constants $\zeta^{^{f}}_{1}$ and $\zeta^{^{f}}_{_{2}}$ depending on $\eta_{1}$ and $\eta_{_{2}}$ (defined in~\eqref{Eq:h_kBound} and~\eqref{Eq:varrho_k-h_kBound}, respectively, and independent of $h$)  such that
    \begin{align}
        \zeta^{^{f}}_{1}h^{^{2}}\left\|\mathbf{c}^{u, N}\right\|^{^{2}}_{_{2}}\|\mathbf{c}\|^{^{2}}_{_{2}}\le\int\limits_{\Omega} \spta{ f,  f }\, \Tilde{u}^{N} \text{d}\mathbf{x}\le  \zeta^{^{f}}_{_{2}}h^{^{2}} \left\|\mathbf{c}^{u, N}\right\|^{^{2}}_{_{2}} \|\mathbf{c}\|^{^{2}}_{_{2}}.
        \label{InEq:mmm_bound}
    \end{align}
\end{remark}
Having obtained the necessary quadratic forms for the terms that constitute \(\mathbf{A}_{f}\), we proceed to estimate their corresponding eigenvalue bounds and ultimately combine them to get the bound for the condition number $k(\mathbf{A}_{f})$. First, we obtain upper and lower bounds for the quadratic form of the first term in~\eqref{Eq:final-A_f} involving the mass matrix \(\mathbf{M}\), by multiplying the $\Tilde{f}$-equivalence of \eqref{InEq:mm_bound-g} by $\left(\frac{1}{\Delta t}+\lambda_{_{f}} \right)$ to obtain
\begin{equation*}
   \left(\frac{1}{\Delta t}+\lambda_{_{f}}\right) \spta{ \mathbf{Mc}^{f},\mathbf{c}^{f}} \le \left(\frac{1}{\Delta t}+\lambda_{_{f}} \right) \zeta^{^{f}}_{1}h^{^{2}}\|\mathbf{c}^{f}\|_{_{2}}^{^{2}},
\end{equation*}
and 
\begin{equation*}
  \left(\frac{1}{\Delta t}+\lambda_{_{f}}\right)\spta{ \mathbf{Mc}^{f},\mathbf{c}^{f}} \ge  \left(\frac{1}{\Delta t}+\lambda_{_{f}} \right) \zeta^{^{f}}_{_{2}}h^{^{2}}\|\mathbf{c}^{f}\|^{^{2}}_{_{2}},
\end{equation*}
 since $\left(\frac{1}{\Delta t}+\lambda_{_{f}} \right)> 0$ and $\mathbf{M}$ is positive definite, where $\zeta^{^{f}}_{1}$ and $\zeta^{^{f}}_{_{2}}$ are the similar constants defined in~\eqref{lemma:mass-mat:stiff-mat} for the $\Tilde{f}$-case. For the second matrix on the right-hand-side in~\eqref{Eq:final-A_f}, we obtain the upper and lower bounds for its quadratic form by multiplying~\eqref{InEq:mmm_bound} by the constant $2p_{_{f}} w_{_{f}}$ together with the $\Tilde{f}$-equivalence of the identity defined in~\eqref{Eq:quadForm-Tc-ECM}
 \begin{align*}
     &2p_{_{f}} w_{_{f}}\spta{\mathbf{T}({c}^{g, N})\mathbf{c}^{f},\mathbf{c}^{f}} \le  2p_{_{f}} w_{_{f}}\zeta^{^{f}}_{_{2}}h^{^{2}} \left\|\mathbf{c}^{u, N}\right\|^{^{2}}_{_{2}} \|\mathbf{c}^{^{f}}\|^{^{2}}_{_{2}},\\
     &2p_{_{f}} w_{_{f}}\spta{\mathbf{T}({c}^{g, N})\mathbf{c}^{f},\mathbf{c}^{f}} \ge  2p_{_{f}} w_{_{f}}\zeta^{^{f}}_{_{2}}h^{^{2}} \left\|\mathbf{c}^{u, N}\right\|^{^{2}}_{_{2}} \|\mathbf{c}^{f}\|^{^{2}}_{_{2}},
 \end{align*}
  which are the needed upper and lower bounds, respectively. For the third matrix on the right-hand-side in~\eqref{Eq:final-A_f} corresponding to the stiffness matrix, we obtain the following upper and lower bounds for its quadratic form by multiplying $\Tilde{f}$-equivalence of~\eqref{InEq:sm_bound-g} by $D_{_{f}}$ to get
\begin{align*}
    0&\le D_{_{f}}\spta{\mathbf{c}^{f},\mathbf{K}\mathbf{c}^{f}}
    ,\;\;\;\;
    D_{_{f}}\spta{\mathbf{c}^{f},\mathbf{K}\mathbf{c}^{f}}\le D_{_{f}}\zeta^{^{f}}_{_{2}}\|\mathbf{c}\|^{^{2}}_{_{2}}.
\end{align*} 
Here we have also used the boundary condition \( \spta{ \nabla \Tilde{f},n} = 0\), 
together with the \(\Tilde{f}\)-equivalence of \eqref{InEq:mm_bound-g} in Lemma~\ref{lemma:mass-mat:stiff-mat}.
Now, we can estimate the bounds for the minimum and maximum eigenvalues of \(\mathbf{A}_{f}\) defined in~\eqref{Eq:final-A_f}, using the summation of the obtained bounds for the constituting matrices:
\begin{align*}
    \frac{\spta{ \mathbf{A}_{f} \mathbf{c}^{^{f}}, \mathbf{c}^{^{f}}}}{\left\|\mathbf{c}^{^{f}}\right\|_{_{2}}^{^{2}}} =& \left(\frac{1}{\Delta t}+\lambda_{_{f}}\right)\frac{\spta{ \mathbf{M} \mathbf{c}^{^{f}}, \mathbf{c}^{^{f}}}}{\left\|\mathbf{c}^{^{f}}\right\|_{_{2}}^{^{2}}} + 2p_{_{f}} w_{_{f}}\frac{\spta{\mathbf{T}(c^{g,N}) \mathbf{c}^{^{f}}, \mathbf{c}^{^{f}}}}{\left\|\mathbf{c}^{^{f}}\right\|_{_{2}}^{^{2}}} + D_{_{f}}\frac{\spta{ \mathbf{K} \mathbf{c}^{^{f}}, \mathbf{c}^{^{f}}}}{\left\|\mathbf{c}^{^{f}}\right\|_{_{2}}^{^{2}}}\\
    \le&\left(\frac{1}{\Delta t}+\lambda_{_{f}} \right) \zeta^{^{f}}_{1}h^{^{2}} + 2p_{_{f}} w_{_{f}}\zeta^{^{f}}_{_{2}}h^{^{2}} \left\|\mathbf{c}^{u, N}\right\|^{^{2}}_{_{2}}  + D_{_{f}}\zeta^{^{f}}_{_{2}}
\end{align*}
 and 
 \begin{align*}
    \frac{\spta{ \mathbf{A}_{f} \mathbf{c}^{^{f}}, \mathbf{c}^{^{f}}}}{\left\|\mathbf{c}^{^{f}}\right\|_{_{2}}^{^{2}}} =& \left(\frac{1}{\Delta t}+\lambda_{_{f}}\right)\frac{\spta{ \mathbf{M} \mathbf{c}^{^{f}}, \mathbf{c}^{^{f}}}}{\left\|\mathbf{c}^{^{f}}\right\|_{_{2}}^{^{2}}} + 2p_{_{f}} w_{_{f}}\frac{\spta{ \mathbf{T}\mathbf{c}^{g, N} \mathbf{c}^{^{f}}, \mathbf{c}^{^{f}}}}{\left\|\mathbf{c}^{^{f}}\right\|_{_{2}}^{^{2}}} + D_{_{f}}\frac{\spta{ \mathbf{K} \mathbf{c}^{^{f}}, \mathbf{c}^{^{f}}}}{\left\|\mathbf{c}^{^{f}}\right\|_{_{2}}^{^{2}}}\\
    \ge&\left(\frac{1}{\Delta t}+\lambda_{_{f}} \right) \zeta^{^{g}}_{1}h^{^{2}} + 2p_{_{f}} w_{_{f}}\zeta^{^{f}}_{_{2}}h^{^{2}} \left\|\mathbf{c}^{u, N}\right\|^{^{2}}_{_{2}}.
\end{align*}
Hence the condition number \(k\left(\mathbf{A}_{f}\right)\) assumes the following bound:
\begin{align*}
k\left(\mathbf{A}_{f}\right)=\frac{\Lambda_{_\max}\left(\mathbf{A}_{f}\right)}{\Lambda_{_{\min}}\left(\mathbf{A}_{f}\right)}
    &\le \frac{\left(\frac{1}{\Delta t}+\lambda_{_{f}} \right) \zeta^{^{g}}_{1}h^{^{2}} + 2p_{_{f}} w_{_{f}}\zeta^{^{f}}_{_{2}}h^{^{2}} \left\|\mathbf{c}^{u, N}\right\|^{^{2}}_{_{2}}  + D_{_{f}}\zeta^{^{g}}_{_{2}}}{\left(\frac{1}{\Delta t}+\lambda_{_{f}} \right) \zeta^{^{g}}_{1}h^{^{2}} + 2p_{_{f}} w_{_{f}}\zeta^{^{f}}_{_{2}}h^{^{2}} \left\|\mathbf{c}^{u, N}\right\|_{_{2}}}\nonumber\\
    &=1 + \frac{D_{_{f}}\zeta^{^{g}}_{_{2}}}{\left(\frac{1}{\Delta t}+\lambda_{_{f}} \right) \zeta^{^{g}}_{1} + 2p_{_{f}} w_{_{f}}\zeta^{^{f}}_{_{2}} \left\|\mathbf{c}^{u, N}\right\|^{^{2}}_{_{2}}} h^{^{-2}}.
\end{align*}

\paragraph{Condition number of the \(m\)-equation.} 
Following here the same argument as for the $f$-equation, we get the corresponding $m$-equation as:
\begin{align}
    \int_{\Omega}\Bigg[\frac{m - m^{N}}{\Delta t} v_{_{m}} + D_{_{m}}\spta{ \nabla m, \nabla v_{_{m}}} + \mu_{_{m}}\spta{ A_{_{m}}^{N}[u],m\nabla v_{_{m}} }\nonumber  
    + \lambda_{_{m}} m v_{_{m}} \\
    - p_{_{m}} g^{N} m\left(1-w_{_{g}} g^{N} - w_{_{f}} f^{N} - w_{_{m}} m - w_{_{e}} e^{N}\right) v_{_{m}}\Bigg]\text{d}\mathbf{x} = 0 \quad \forall v_{_{m}},
    \label{Eq:mac:timeDisc}        
\end{align}
where 
\begin{align}
    A_{_{m}}^{N}[u](\mathbf{x}) = \frac{1}{R}{\int\limits_{\textbf{B}_{_{\|\cdot\|_{_{\infty}}}}\!\!(\textbf{0},R)}}\!\!\!\mathrm{K}(\|\mathbf{y}\|_{_{2}})  \mathbf{n}(\mathbf{y})\left(\!\!\left(1-\rho\left(u^{N}\right)\!\!\right)^{+}\,{\Gamma}_{_m}\left(u^{N}\right)\!\!\right)(\mathbf{x}+\mathbf{y},t)\, \text{d}\mathbf{y}.
    \label{Eq:nonlocal_A_m}
\end{align}
 
We see that the fifth term on the left-hand-side of \eqref{Eq:mac:timeDisc} (which we denote as $T^{m}_{_{5}}$) is non-linear in the unknown $m$. For that reason, we linearize it with respect to the unknown values $m$ evaluated at the known (\textit{i.e}., at time $N$) values of the other components $g, f,$ and $e$: 
\begin{equation}
    \frac{d T^{m}_{_{5}}}{d m} = 2 p_{_{m}}g^{N}w_{_{m}} m - p_{_{m}}g^{N}\left(1 \!-\! w_{_{g}}g^{N}\!-\! w_{_{f}} f^{N}\!-\! w_{_{e}} e^{N}\right).
    \label{Eq:linearizing_logistic_m}
\end{equation}
Furthermore, just as we did for the $f$-case, we approximate some terms in the integrand in~\eqref{Eq:nonlocal_A_m} as follows:
\begin{equation*}
    0\le \left(1-\rho\left(u^{N}(\mathbf{x}+\mathbf{y},t)\right)\right)^{^{\!\!+}} = \left(\!1-\!(w_{_{g}} g^{N}\!\!\! +\! w_{_{f}} f^{N}\!\!\! + \!w_{_{m}} m^{N}\!\!\! +\!w_{_{e}} e^{N})(\mathbf{x}+\mathbf{y},t)\!\right)^{^{\!\!+}} \le 1,
\end{equation*} 
and 
\begin{align*}
   {\Gamma}_{_m}\left(u^{N}\right)(\mathbf{x}+\mathbf{y},t) &= \left(\!\!\frac{e^{N}\!\! + \!g^{N}}{1\!+\!e^{N}\!\! + \!g^{N}}\!\!\right)\!\left(\!S_{_{mf}}^{\max} f^{N} \!\!+S_{_{mm}}^{\max} m^{N} \!\!+S_{_{me}}^{\max} e^{N}\!\right)(\mathbf{x}+\mathbf{y},t)\\
    &\le S_{_{mf}}^{\max} f^{N}(\mathbf{x}+\mathbf{y},t) \!\!+S_{_{mm}}^{\max} m^{N}(\mathbf{x}+\mathbf{y},t) \!\!+S_{_{me}}^{\max} e^{N}(\mathbf{x}+\mathbf{y},t),
\end{align*}
where we have used the fact that $\left(e^{N}\!\! + \!g^{{N}}\right)/\left(1\!+\!e^{{N}}\!\! + \!g^{{N}}\right)\le 1$. Furthermore, approximating 
each $ j^{N}(\mathbf{x}+\mathbf{y},t), j\in\{f, m, e\}$ with its Taylor expansion and assuming that the higher order terms are small enough to be ignored, we have
\begin{align*}
   j^{N}(\mathbf{x}+\mathbf{y},t) \approx j^{N}\!(\mathbf{x},t) + \spta{\! \nabla j^{N}\!(\mathbf{x},t), \,\mathbf{y}\!},\;\;\;j\in\{f, m, e\}.
\end{align*}
Now, substituting the unit normal defined in \eqref{surf:normal:def} and the kernel defined in \eqref{kernel:Gaussian} together with the above approximations, following the exact steps as in the $f$-case (see~\eqref{Eq:nonlocal:Approx-f}), the non-local term $A_{_{m}}^{N}[u](\mathbf{x})$ can now be approximated by $\Tilde{A}_{_{m}}^{N}[u](\mathbf{x})$ that is given by
\begin{align}
    \Tilde{A}_{_{m}}^{N}[u](\mathbf{x}): 
      =& \frac{C_{_{\rho, R}}}{R}\left(S_{_{mf}}^{^{\max}} \nabla m^{N}(\mathbf{x},t) + S_{_{mm}}^{^{\max}}\nabla m^{N}(\mathbf{x},t)+S_{_{me}}^{^{\max}} \nabla e^{N}(\mathbf{x},t)\right).
       \label{Eq:nonlocal:Approx-m}
\end{align}
Here 
\[ 
C_{_{\sigma, R}} = \frac{-\rho\,\text{ erf} \left(\frac{R\sqrt{2}}{2\rho}\right)\left[\sqrt{2\pi} \,\,e^{-R^{^{2}}/2\rho^{^{2}}} R - \pi \text{ erf}\left(\frac{R\sqrt{2}}{2\rho}\rho\right)\right]}{\pi},
\] 
since \[\int\limits_{\textbf{B}_{_{\|\cdot\|_{_{\infty}}}}\!\!(\textbf{0},R)}\!\!\!\!\!\!e^{^{-{\|\mathbf{y}\|_{_{2}}^{^{2}}\big/2\sigma^{^{2}}}}} y_{i} y_{j}\;\text{d}\mathbf{x} =
\begin{cases} C_{_{\rho, R}}, & i=j\\
    0, &i\neq j
\end{cases}  
\quad i,j\in\{1,2\}.
\]

Now that we are done with the various approximations, we substitute the approximation of \(T^{m}_{5}\) in~\eqref{Eq:linearizing_logistic_m} and the approximation of the non-local term $ A_{_{m}}^{N}[u](\mathbf{x})$ in~\eqref{Eq:nonlocal:Approx-m} into \eqref{Eq:mac:timeDisc} and integrate by parts to obtain
\begin{align*}
     &\int_{\Omega}\Bigg[\frac{m - m^{N}}{\Delta t} v_{_{m}} + D_{_{m}}\spta{ \nabla m, \nabla v_{_{m}}} 
     - \frac{\mu_{_{m}} C_{_{\rho, R}}}{R}\left(S_{_{mf}}^{^{\max}}\!  m  \spta{ \nabla f^{N},\nabla  v_{_{m}}  }+\right.
     S_{_{mm}}^{^{\max}}\! m  \spta{ \nabla m^{N},\nabla  v_{_{m}}  }+\\
     &S_{_{me}}^{^{\max}}\! m  \left.\spta{ \nabla e^{N},\nabla  v_{_{m}}  }\right)  
      +\lambda_{_{m}} m v_{_{m}} +  2 p_{_{m}}g^{N}w_{_{m}} m - p_{_{m}}g^{N}\left(1 \!-\! w_{_{g}}g^{N}\!-\! w_{_{f}} f^{N}\!-\! w_{_{e}} e^{N}\right) v_{_{m}}\Bigg]\text{d}\mathbf{x} = 0 \quad\forall v_{_{m}}.
\end{align*}
Here we have used the boundary condition in~\eqref{EqBC:ExisUniq}, which reduces simply to \(\spta{ \nabla f, n}= \spta{ \nabla m, n}=0\) upon the approximation of the non-local term~\(A_{_{m}}^{N}[u](x)\) in \eqref{Eq:nonlocal:Approx-m}. Furthermore, upon discretization in space we have 
\begin{align*}
    &\sum_{\tau}^{l}\Bigg[\left(\frac{1}{\Delta t}+\lambda_{_{m}} \right)\int_{\Omega}\psi_{\tau}\psi_{o}\,\text{d}\mathbf{x} + \sum_{q}^{l}2p_{_{m}} w_{_{m}}c_{q}^{g,N}\int_{\Omega}\psi_{q}\psi_{\tau}\psi_{o}\,\text{d}\mathbf{x} + D_{_{m}} \int_{\Omega} \spta{ \nabla\psi_{\tau}, \nabla\psi_{o}} \text{d}\mathbf{x}   \,\,-\nonumber\\
    &\frac{\mu_{_{m}} C_{_{\rho, R}}}{R}\sum_{q}\Bigg[\left(S_{_{mf}}^{^{\max}} c_{q}^{f,N} +S_{_{mm}}^{^{\max}} c_{q}^{m,N}+S_{_{me}}^{^{\max}} c_{q}^{e,N}\right) 
    \int_{\Omega}\spta{ \nabla\psi_{q}, \nabla \psi_{o}} \psi_{\tau}\text{d}\mathbf{x}\Bigg]\Bigg]{c}_{\tau}^{m}\\  
    &=\int\limits_{\Omega}\sum_{\tau = 1}^{l} \Bigg[ p_{_{m}} c_{\tau}^{g, N}\psi_{\tau} \left(1 - \sum_{s = 1}^{l}\psi_{s} \left(c_{s}^{g,N},c_{s}^{f,N},c_{s}^{e,N} \right)\right) +\frac{1}{\Delta t} c_{\tau}^{m,N}\psi_{\tau}\Bigg]\psi_{o}\text{d}\textbf{x}\quad\forall o\in\{1,\dots, l\}.
\end{align*}
We observe that this is already in the general form
\begin{equation*}
    \mathbf{A}_{m} \mathbf{c}^{m}= \mathbf{F}_{_m},
\end{equation*}
where the entries of $\mathbf{A}_{m}\in\mathbb{R}^{l\times l}$ are given as
\begin{align*}
    (\mathbf{A}_{m})_{\tau o} = &\left(\frac{1}{\Delta t}+\lambda_{_{m}} \right)\int_{\Omega}\psi_{\tau}\psi_{o}\,\text{d}\mathbf{x} + \sum_{q}^{l}2p_{_{m}} w_{_{m}}c_{q}^{g,N}\int_{\Omega}\psi_{q}\psi_{\tau}\psi_{o}\,\text{d}\mathbf{x} + D_{_{m}} \int_{\Omega} \spta{ \nabla\psi_{\tau}, \nabla\psi_{o}} \text{d}\mathbf{x}   \,\,-\nonumber\\
    &\frac{\mu_{_{m}} C_{_{\rho, R}}}{R}\sum_{q}\Bigg[\left(S_{_{mf}}^{^{\max}} c_{q}^{f,N} +S_{_{mm}}^{^{\max}} c_{q}^{m,N}+S_{_{me}}^{^{\max}} c_{q}^{e,N}\right)\int_{\Omega}\spta{ \nabla\psi_{q}, \nabla \psi_{o}} \psi_{\tau}\;\text{d}\mathbf{x}\Bigg],
\end{align*}
and the entries of $(\mathbf{F}_{m})_{o}\in\mathbb{R}^{l}$ as
\begin{equation*}
    (\mathbf{F}_{m})_{o} = \int\limits_{\Omega}\sum_{\tau = 1}^{l} \Bigg[ p_{_{m}} c_{\tau}^{g, N}\psi_{\tau} \left(1 - \sum_{s = 1}^{l}\psi_{s} \left(c_{s}^{g,N},c_{s}^{f,N},c_{s}^{e,N} \right)\right)
    +\frac{1}{\Delta t} c_{\tau}^{e,N}\psi_{\tau}\Bigg]\psi_{o}\;\text{d}\textbf{x}.
\end{equation*}
We can write \(\mathbf{A}_{m}\) in the compact form
\begin{align*}
   \mathbf{A}_{m} = \left(\frac{1}{\Delta t}+\lambda_{_{m}}\right) \mathbf{M} + 2p_{_{m}} w_{_{m}}\mathbf{T}({c}^{g, N}) + D_{_{m}} \mathbf{K} -
    \frac{\mu_{_{m}} C_{_{\rho, R}}}{R}\left(S_{_{mf}}^{^{\max}} \mathbf{H}({c}^{f,N})+\left.S_{_{mm}}^{^{\max}}\mathbf{H}({c}^{m,N})\right)+S_{_{me}}^{^{\max}} \mathbf{H}({c}^{e,N})\right),
\end{align*}
where $\mathbf{M}$ is the usual mass matrix, 
\begin{align*}
\mathbf{T}\left(c^{g, N}\right) =  \begin{bmatrix}
        \sum_{q}^{l} c_{q}^{g, N} \int\limits_{\Omega} \psi_{q} \psi_{1} \psi_{1}d\mathbf{x} &\ldots &  \sum_{q}^{l} c_{q}^{g, N} \int\limits_{\Omega} \psi_{q} \psi_{1} \psi_{l}d\mathbf{x}\\
        \vdots & & \vdots\\
          \sum_{q}^{l} c_{q}^{g, N} \int_{\Omega} \psi_{q} \psi_{l} \psi_{1}d\mathbf{x} &\ldots &  \sum_{q}^{l} c_{q}^{g, N} \int\limits_{\Omega} \psi_{q} \psi_{l} \psi_{l}d\mathbf{x}\end{bmatrix},
\end{align*}
and 
\begin{align*}
    \mathbf{H}({c}^{i,N}) =  
    \begin{bmatrix}
        \sum_{q=1}^{l}c^{i,N}_{q}\int\limits_{\Omega} \spta{ \nabla\psi_{q}, \nabla\psi_{1}} \psi_{1}d\mathbf{x} &\ldots & \sum_{q=1}^{l}c^{i,N}_{q}\int\limits_{\Omega} \spta{ \nabla\psi_{q}, \nabla\psi_{l}} \psi_{1}d\mathbf{x}\\
        \vdots & & \vdots\\
        \sum_{q=1}^{l}c^{i,N}_{q}\int_{\Omega} \spta{ \nabla\psi_{q}, \nabla\psi_{1}} \psi_{l}d\mathbf{x} &\ldots &  \sum_{q=1}^{l}c^{i,N}_{q}\int\limits_{\Omega} \spta{ \nabla\psi_{q}, \nabla\psi_{l}} \psi_{l}d\mathbf{x}
    \end{bmatrix}, \quad i\in\{f, m, e\}.
\end{align*}
To estimate the condition number of \(\mathbf{A}_{m}\in\mathbb{R}^{l\times l}\), just as for the \(f\)-equation, we will first derive its quadratic form \(\spta{ \mathbf{A}_{m} \mathbf{c}^{m},\mathbf{c}^{m}}\) by deriving the quadratic form of the matrices that constitute it. We observe that the matrix \(\mathbf{H}({c}^{i,N})\), which is derived from the discretization of the adhesion flux term, is not symmetric. Hence, we amend the numerical scheme to handle it. For more details, we refer the reader to Remark~\ref{rem:Af} (as well as Remark~\ref{rem:Am} below).

For an arbitrary \(\Tilde{m} = \sum_{\tau =1}^{l} c_{\tau}^{m} \psi_{\tau}\) and \(\Tilde{u}^{N} = \sum_{\tau =1}^{l} c^{u,N}_{\tau} \psi_{\tau}\), where \(\Tilde{u}\) could be any of the variables \(\{g,f,m,e\}\), following the same steps as in~\eqref{Eq:quadFormMatrix-StiffMatUn} leads to
{\footnotesize
\begin{align*}
    \int\limits_{\Omega}\spta{\spta{ \nabla \Tilde{m} , \nabla \Tilde{m}}\, \Tilde{u}^{N}}\:\text{d}\mathbf{x} =\:&
     \sptb{
    \begin{bmatrix}
        \sptb{
        \begin{bmatrix}
            \int\limits_{\Omega} \spta{\nabla \psi_{1}, \nabla\psi_{1}} \psi_{1}d\mathbf{x} &\ldots&  \int\limits_{\Omega} \spta{\nabla\psi_{1} , \nabla\psi_{1}} \psi_{l}d\mathbf{x}\\
            \vdots & & \vdots\\
            \int\limits_{\Omega} \spta{\nabla\psi_{1}, \nabla\psi_{l}} \psi_{1}d\mathbf{x} &\ldots & \int\limits_{\Omega} \spta{\nabla\psi_{1}, \nabla\psi_{l}} \psi_{l}d\mathbf{x}
        \end{bmatrix}
        \mathbf{c}^{m},\mathbf{c}^{m}}\\
        \vdots\\
        \sptb{
        \begin{bmatrix}
            \int\limits_{\Omega} \spta{\nabla\psi_{l}, \nabla\psi_{1}} \psi_{1}d\mathbf{x} &\ldots&  \int\limits_{\Omega} \spta{\nabla\psi_{l}, \nabla\psi_{1}} \psi_{l}d\mathbf{x}\\
            \vdots & & \vdots\\
            \int\limits_{\Omega} \spta{\nabla\psi_{l}, \nabla \psi_{l}} \psi_{1}d\mathbf{x} &\ldots & \int\limits_{\Omega} \spta{\nabla\psi_{l}, \psi_{l} }\psi_{l}d\mathbf{x}
        \end{bmatrix}
        \mathbf{c}^{m},\mathbf{c}^{m}} 
    \end{bmatrix}, \mathbf{c}^{u, N} }\\
    =& \sptb{
    \begin{bmatrix}
        \int\limits_{\Omega} \spta{\nabla \psi_{1}, \nabla\psi_{1}} \psi_{1}d\mathbf{x} &\ldots&  \int\limits_{\Omega} \spta{\nabla\psi_{1} , \nabla\psi_{l}} \psi_{1}d\mathbf{x}\\
        \vdots & & \vdots\\
        \int\limits_{\Omega} \spta{\nabla\psi_{1}, \nabla\psi_{1}} \psi_{l}d\mathbf{x} &\ldots & \int\limits_{\Omega} \spta{\nabla\psi_{1}, \nabla\psi_{l}} \psi_{l}d\mathbf{x}
    \end{bmatrix}
    \mathbf{c}^{m},\mathbf{c}^{m}} \:\:c_{1}^{u,N} \\ &+\ldots+\sptb{
    \begin{bmatrix}
        \int\limits_{\Omega} \spta{\nabla\psi_{l}, \nabla\psi_{1}} \psi_{1}d\mathbf{x} &\ldots&  \int\limits_{\Omega} \spta{\nabla\psi_{l}, \nabla\psi_{l}} \psi_{1}d\mathbf{x}\\
        \vdots & & \vdots\\
        \int\limits_{\Omega} \spta{\nabla\psi_{l}, \nabla \psi_{1}} \psi_{l}d\mathbf{x} &\ldots & \int\limits_{\Omega} \spta{\nabla\psi_{l}, \psi_{l} }\psi_{l}d\mathbf{x}
    \end{bmatrix}\mathbf{c}^{m},\mathbf{c}^{m}}\:\: c_{l}^{u,N}  \\[1.0ex]
    =&\sptb{
    \begin{bmatrix}
        \sum_{q}^{l} c_{q}^{u, N} \int\limits_{\Omega} \spta{\nabla\psi_{q},\nabla \psi_{1} }\psi_{1}d\mathbf{x} &\ldots &  \sum_{q}^{l} c_{q}^{u, N} \int\limits_{\Omega} \spta{\nabla\psi_{q} ,\nabla\psi_{l} } \psi_{1}d\mathbf{x}\\[1.0ex]
        \vdots & & \vdots\\
        \sum_{q}^{l} c_{q}^{u, N} \int\limits_{\Omega} \spta{\nabla\psi_{q}, \nabla \psi_{1} }\psi_{l}d\mathbf{x} &\ldots &  \sum_{q}^{l} c_{q}^{u, N} \int\limits_{\Omega} \spta{\nabla\psi_{q}, \nabla \psi_{l} }\psi_{l}d\mathbf{x}
    \end{bmatrix}\mathbf{c}^{m},\mathbf{c}^{m}}\\
    =&\spta{ \mathbf{H}({c}^{u,N})\mathbf{c}^{m},\mathbf{c}^{m} } 
\end{align*}}

\begin{remark}
\label{rem:Am}
    From the resulting matrix \(\mathbf{H}({c}^{u,N})\), corresponding to the nonlocal adhesive term, it is easy to see that we obtain a non-symmetric matrix. Hence, obtaining an estimate of the condition number for \(\mathbf{A}_{m}\) using the Rayleigh quotient approach becomes complicated, since we only use this approach when the ratio of the maximum to minimum eigenvalues coincide with the singular values (\textit{i.e}., when considering symmetric matrices). To address this issue, just as in the $f$-case, we amend the numerical scheme as detailed in Remark~\ref{rem:Af}. Note also that the numerical simulations in Figure~\ref{Fig-L2NormDifference} suggest that there is no significant difference between our initial and amended numerical schemes. 
\end{remark}
\noindent Just as for the $f$-equation, Remark \ref{rem:Am} reduces \(\mathbf{A}_{m}\) to 
\begin{align}
    \mathbf{A}_{m} = \left(\frac{1}{\Delta t}+\lambda_{_{m}}\right) \mathbf{M} + 2p_{_{m}} w_{_{m}}\mathbf{T}({c}^{g, N}) + D_{_{m}} \mathbf{K}.
    \label{Eq:final_Am}
\end{align}
Moreover, \(\mathbf{F}_{m}\) now has entries
\begin{align*}
     (\mathbf{F}_{m})_{o} = \int\limits_{\Omega}\sum_{\tau = 1}^{l} \Bigg[ p_{_{m}} c_{\tau}^{g, N}\psi_{\tau} \left(1 - \sum_{s = 1}^{l}\psi_{s} \left(c_{s}^{g,N},c_{s}^{f,N},c_{s}^{e,N} \right)\right)
    +\frac{1}{\Delta t} c_{\tau}^{e,N}\psi_{\tau}\Bigg]\psi_{o}\;\text{d}\textbf{x} - \\
    \frac{\mu_{_{m}} C_{_{\rho, R}}}{R}\sum_{q}\Bigg[\left(S_{_{mf}}^{^{\max}} c_{q}^{f,N} +S_{_{mm}}^{^{\max}} c_{q}^{m,N}+S_{_{me}}^{^{\max}} c_{q}^{e,N}\right)\int_{\Omega}\spta{ \nabla\psi_{q}, \nabla \psi_{o}} \psi_{\tau}\:\text{d}\mathbf{x}\Bigg].
\end{align*}

\begin{remark}
\label{Rem:Mac}
    Given Remark~\ref{Rem:Fib}, we have that for an arbitrary $\Tilde{m}= \sum_{\tau}^{l}c_{\tau}^{m}\psi_{\tau}$ and $\Tilde{u}^{N}= \sum_{\tau}^{l}c_{\tau}^{u, N}\psi_{\tau},\: u\in~\{g,f,m,e\}$ there exist constants $\zeta^{^{m}}_{1}$ and $\zeta^{^{m}}_{_{2}}$ depending on $\eta_{1}$ and $\eta_{_{2}}$ (defined in~\eqref{Eq:h_kBound} and~\eqref{Eq:varrho_k-h_kBound}, respectively, and independent of $h$) such that
    \begin{align}
    \zeta^{^{m}}_{1}h^{^{2}}\left\|\mathbf{c}^{u, N}\right\|^{^{2}}_{_{2}}\|\mathbf{c}\|^{^{2}}_{_{2}}\le\int\limits_{\Omega} \spta{ m,  m }\, \Tilde{u}^{N} \text{d}\mathbf{x}\le  \zeta^{^{m}}_{_{2}}h^{^{2}} \left\|\mathbf{c}^{u, N}\right\|^{^{2}}_{_{2}} \|\mathbf{c}\|^{^{2}}_{_{2}}\label{InEq:mmm_bound-mac}.
    \end{align}
\end{remark}
Now that we have obtained the necessary quadratic forms for the terms that constitute \(\mathbf{A}_{m}\), we proceed to estimate their corresponding eigenvalue bounds and ultimately combine them to get the bound for the condition number $k(\mathbf{A}_{m})$. First, we obtain upper and lower bounds for the quadratic form of the first term in~\eqref{Eq:final_Am} involving the mass matrix \(\mathbf{M}\), by multiplying the $\Tilde{m}$-equivalence of \eqref{InEq:mm_bound-g} by $\left(\frac{1}{\Delta t}+\lambda_{_{m}} \right)$ to obtain
\begin{equation}
   \left(\frac{1}{\Delta t}+\lambda_{_{m}}\right) \spta{ \mathbf{Mc}^{m},\mathbf{c}^{m}} \le \left(\frac{1}{\Delta t}+\lambda_{_{m}} \right) \zeta^{^{m}}_{1}h^{^{2}}\|\mathbf{c}^{m}\|_{_{2}}^{^{2}},
    \label{Ineq:upbnd_massmatrix_Mac}
\end{equation}
and 
\begin{equation}
  \left(\frac{1}{\Delta t}+\lambda_{_{m}}\right)\spta{ \mathbf{Mc}^{m},\mathbf{c}^{m}} \ge  \left(\frac{1}{\Delta t}+\lambda_{_{m}} \right) \zeta^{^{m}}_{_{2}}h^{^{2}}\|\mathbf{c}^{m}\|^{^{2}}_{_{2}}.
    \label{Ineq:lowbnd_massmatrix_mac}
\end{equation}
 Here, $\left(\frac{1}{\Delta t}+\lambda_{_{m}} \right)> 0$ and $\mathbf{M}$ is positive definite. 
 For the second matrix on the right-hand-side in~\eqref{Eq:final_Am}, we obtain the following upper and lower bounds for its quadratic form by multiplying~\eqref{InEq:mmm_bound-mac} by the constant $2p_{_{m}} w_{_{m}}$ together with the $\Tilde{m}$-equivalence of the identity defined in~\eqref{Eq:quadForm-Tc-ECM}:
 \begin{align*}
     &2p_{_{m}} w_{_{m}}\spta{\mathbf{T}({c}^{g, N})\mathbf{c}^{m},\mathbf{c}^{m}} \le  2p_{_{m}} w_{_{m}}\zeta^{^{m}}_{_{2}}h^{^{2}} \left\|\mathbf{c}^{g, N}\right\|^{^{2}}_{_{2}} \|\mathbf{c}^{^{m}}\|^{^{2}}_{_{2}},\\
     &2p_{_{m}} w_{_{m}}\spta{\mathbf{T}({c}^{g, N})\mathbf{c}^{m},\mathbf{c}^{m}} \ge  2p_{_{m}} w_{_{m}}\zeta^{^{m}}_{_{2}}h^{^{2}} \left\|\mathbf{c}^{g, N}\right\|^{^{2}}_{_{2}} \|\mathbf{c}^{m}\|^{^{2}}_{_{2}}.
 \end{align*}
 For the third term of \(\mathbf{A}_{m}\) as shown in \eqref{Eq:final_Am}, corresponding to the stiffness matrix, we derive the following upper and lower bounds by multiplying \eqref{Eq:quadForm-StiffMatrix} by $D_{_{m}}$ for the case \(\Tilde{m}=\sum_{\tau=1}^{l}c^{m}_{\tau}\psi_{\tau}\):
\begin{align}
    0&\le D_{_{m}}\spta{\mathbf{c}^{m},\mathbf{K}\textbf{c}^{m}}\label{Ineq:lowerbnd-stiffmatrix-mac},\\
    D_{_{m}}\spta{\mathbf{c}^{m},\mathbf{K}\textbf{c}^{m}}&\le D_{_{m}}\zeta^{^{g}}_{_{2}}\|\mathbf{c}\|^{^{2}}_{_{2}},
    \label{Ineq:upperbnd-stiffmatrix-mac}
\end{align} 
where we have also used the boundary condition \( \spta{ \nabla \Tilde{m},n} = 0\)  
together with Lemma~\ref{lemma:mass-mat:stiff-mat} for \(\Tilde{m}\).

Now, we can estimate the bounds for the minimum and maximum eigenvalues of \(\mathbf{A}_{f}\) defined in~\eqref{Eq:final_Am}, using the summation of the bounds obtained for the constituting matrices in \eqref{Ineq:upbnd_massmatrix_Mac} - \eqref{Ineq:upperbnd-stiffmatrix-mac} as follows:
\begin{align*}
    \frac{\spta{ \mathbf{A}_{m} \mathbf{c}^{^{m}}, \mathbf{c}^{^{m}}}}{\left\|\mathbf{c}^{^{m}}\right\|_{_{2}}^{^{2}}} =& \left(\frac{1}{\Delta t}+\lambda_{_{m}}\right)\frac{\spta{ \mathbf{M} \mathbf{c}^{^{m}}, \mathbf{c}^{^{m}}}}{\left\|\mathbf{c}^{^{m}}\right\|_{_{2}}^{^{2}}} + 2p_{_{m}} w_{_{m}}\frac{\spta{\mathbf{T}(c^{g,N}) \mathbf{c}^{^{m}}, \mathbf{c}^{^{m}}}}{\left\|\mathbf{c}^{^{m}}\right\|_{_{2}}^{^{2}}} + D_{_{m}}\frac{\spta{ \mathbf{K} \mathbf{c}^{^{m}}, \mathbf{c}^{^{m}}}}{\left\|\mathbf{c}^{^{m}}\right\|_{_{2}}^{^{2}}}\\
    \le&\left(\frac{1}{\Delta t}+\lambda_{_{m}} \right) \zeta^{^{m}}_{1}h^{^{2}} + 2p_{_{m}} w_{_{m}}\zeta^{^{m}}_{_{2}}h^{^{2}} \left\|\mathbf{c}^{u, N}\right\|^{^{2}}_{_{2}}  + D_{_{m}}\zeta^{^{m}}_{_{2}},
\end{align*}
 and 
 \begin{align*}
    \frac{\spta{ \mathbf{A}_{m} \mathbf{c}^{^{m}}, \mathbf{c}^{^{m}}}}{\left\|\mathbf{c}^{^{m}}\right\|_{_{2}}^{^{2}}} =& \left(\frac{1}{\Delta t}+\lambda_{_{m}}\right)\frac{\spta{ \mathbf{M} \mathbf{c}^{^{m}}, \mathbf{c}^{^{m}}}}{\left\|\mathbf{c}^{^{m}}\right\|_{_{2}}^{^{2}}} + 2p_{_{m}} w_{_{m}}\frac{\spta{ \mathbf{T}(c^{g,N}) \mathbf{c}^{^{m}}, \mathbf{c}^{^{m}}}}{\left\|\mathbf{c}^{^{m}}\right\|_{_{2}}^{^{2}}} + D_{_{m}}\frac{\spta{ \mathbf{K} \mathbf{c}^{^{m}}, \mathbf{c}^{^{m}}}}{\left\|\mathbf{c}^{^{m}}\right\|_{_{2}}^{^{2}}}\\
    \ge&\left(\frac{1}{\Delta t}+\lambda_{_{m}} \right) \zeta^{^{m}}_{1}h^{^{2}} + 2p_{_{m}} w_{_{m}}\zeta^{^{m}}_{_{2}}h^{^{2}} \left\|\mathbf{c}^{g, N}\right\|^{^{2}}_{_{2}}.
\end{align*}
Hence the condition number \(k\left(\mathbf{A}_{m}\right)\) assumes the following bound:
\begin{align*}
k\left(\mathbf{A}_{m}\right):=\frac{\Lambda_{_\max}\left(\mathbf{A}_{m}\right)}{\Lambda_{_{min}}\left(\mathbf{A}_{m}\right)}
    &\le \frac{\left(\frac{1}{\Delta t}+\lambda_{_{m}} \right) \zeta^{^{m}}_{1}h^{^{2}} + 2p_{_{m}} w_{_{m}}\zeta^{^{m}}_{_{2}}h^{^{2}} \left\|\mathbf{c}^{g, N}\right\|^{^{2}}_{_{2}}  + D_{_{m}}\zeta^{^{m}}_{_{2}}}{\left(\frac{1}{\Delta t}+\lambda_{_{m}} \right) \zeta^{^{m}}_{1}h^{^{2}} + 2p_{_{m}} w_{_{m}}\zeta^{^{m}}_{_{2}}h^{^{2}} \left\|\mathbf{c}^{g, N}\right\|^{^{2}}_{_{2}}}\nonumber\\
    & = 1 + \frac{D_{_{m}}\zeta^{^{m}}_{_{2}}}{\left(\frac{1}{\Delta t}+\lambda_{_{m}} \right) \zeta^{^{m}}_{1} + 2p_{_{m}} w_{_{m}}\zeta^{^{m}}_{_{2}} \left\|\mathbf{c}^{g, N}\right\|^{^{2}}_{_{2}}} h^{^{-2}}
\end{align*}

\subsection{Condition number of the whole system.}
As shown above, the estimated condition numbers of the \(g\)-, \(f\)-, \(m\)-, and \(e\)-equations at time \(N+1\) depend on the parameter values, the initial condition of the growth factor, and the spatial grid size—except for the ECM equation, whose condition number remains constant.
The condition number of the overall system, \(k(\mathbf{A}_{g,f,m,e})\), at time \(N+1\), is therefore estimated to satisfy the following bound:
\begin{align}
   k(\mathbf{A}_{g,f,m,e}) := \max&\{k(\mathbf{A}_{g}),  k(\mathbf{A}_{f}),
   k(\mathbf{A}_{m}),k(\mathbf{A}_{e})\}\label{Eq:SystCondNbr} \\
   \le \max &\Bigg\{  
   1 +\frac{D_{_{g}}\zeta^{^{g}}_{_{2}}}
   {\left(\frac{1}{\Delta t}+\lambda_{_{g}}\right)\zeta^{^{g}}_{_{1}}} h^{-2},\nonumber \\
   &\quad 1 + \frac{D_{_{f}}\zeta^{^{f}}_{_{2}}}
   {\left(\frac{1}{\Delta t}+\lambda_{_{f}} \right) \zeta^{^{f}}_{1} 
   + 2p_{_{f}} w_{_{f}}\zeta^{^{f}}_{_{2}} 
   \left\|\mathbf{c}^{g, N}\right\|^{2}_{2}} h^{-2}, \nonumber\\
   &\quad 1 + \frac{D_{_{m}}\zeta^{^{m}}_{_{2}}}
   {\left(\frac{1}{\Delta t}+\lambda_{_{m}} \right) \zeta^{^{m}}_{1} 
   + 2p_{_{m}} w_{_{m}}\zeta^{^{m}}_{_{2}} 
   \left\|\mathbf{c}^{g, N}\right\|^{2}_{2}} h^{-2},\quad C 
   \Bigg\}.\nonumber
\end{align}
This implies that, for different choices of parameter values governing the dynamics of the constituent variables, the system exhibits different condition numbers. That is, the individual condition numbers \(k(\mathbf{A}_{g})\), \(k(\mathbf{A}_{f})\), \(k(\mathbf{A}_{m})\), and \(k(\mathbf{A}_{e})\) may dominate the system depending on the specific parameter configuration. We discuss the various possible cases in Tables~\ref{tab:condition-numbers} and \ref{tab:condition-numbers-2}.

\begin{table}[htpb]
\caption{Analysis of Condition Numbers in Biological Models (Part 1).}
\centering
\renewcommand{\arraystretch}{1.5}
\begin{tabular}{@{}c p{3.5cm} p{2cm} p{4cm} p{4cm}@{}}
\toprule
\textbf{Case} & \textbf{Parameter Values} & \textbf{Dominant Term} & \textbf{Biological/Numerical Implication} & \textbf{Condition Number Behavior} \\
\midrule
\textbf{A} & 
\vspace{-0.6cm} \begin{itemize}[]
  \item High $ D_g $
  \item Small $ \lambda_g $, $ \Delta t $, $ h $
\end{itemize} & 
$ k(\mathbf{A}_g) $ & 
\vspace{-0.6cm} \begin{itemize}[]
  \item Growth factor diffuses rapidly.
  \item Slow growth factor decay, fine spatial and temporal resolution.
\end{itemize} & 
\vspace{-0.6cm} \begin{itemize}[]
  \item System becomes stiff due to rapid growth factor spread and low decay.
  \item Growth factor dynamics dominate the system's stability.
\end{itemize} \\

\textbf{B} & 
\vspace{-0.6cm} \begin{itemize}[]
  \item High $ D_f $, low $ \lambda_f $
  \item Low $ \left\| \mathbf{c}^{g,N} \right\|_2^{2}$ or small $ p_f$, $ w_f $
\end{itemize} & 
$ k(\mathbf{A}_f) $ & 
\vspace{-0.6cm} \begin{itemize}[]
  \item Fast fibroblast diffusion and low decay rate.
  \item Low proliferation of fibroblast due to weak growth factor signal caused by low growth factor concentration or low proliferation rate.
\end{itemize} & 
\vspace{-0.6cm} \begin{itemize}[]
  \item Fibroblast dynamics dominate due to high diffusion and slow proliferation. The extra damping term in $k(\mathbf{A}_f)$ is small due to low  $ \left\| \mathbf{c}^{g,N} \right\|_2^{2}$ or small $ p_f$, $ w_f $
\end{itemize} \\

\textbf{C} & 
\vspace{-0.6cm} \begin{itemize}[]
  \item High $ D_m $, low $ \lambda_m $
  \item Low $ \left\| \mathbf{c}^{g,N} \right\|_2^{2}$ or small $ p_m$, $ w_m $
\end{itemize} & 
$ k(\mathbf{A}_m) $ & 
\vspace{-0.6cm} \begin{itemize}[]
  \item Fast macrophage diffusion and low decay rate.
  \item Low proliferation of macrophages due to weak growth factor signal caused by low growth factor concentration or low proliferation rate.
\end{itemize} & 
\vspace{-0.6cm} \begin{itemize}[]
  \item Macrophage dynamics dominate due to high diffusion and slow proliferation. The extra damping term in $k(\mathbf{A}_m)$ is small due to low  $ \left\| \mathbf{c}^{g,N} \right\|_2^{2}$ or small $ p_m$, $ w_m $
\end{itemize} \\

\textbf{D} & 
\vspace{-0.6cm} \begin{itemize}[]
  \item Moderate \(D_{i}, \lambda_{i}\), \(i~\in~\{g,f,m\}\), \( p_{f}, p_{m}\)
  \item Coarse spatial grid (large $h$)
  \item Moderately large $\Delta t$
\end{itemize} & 
$ k(\mathbf{A}_e) $ & 
\vspace{-0.6cm} \begin{itemize}[]
  \item Balanced ECM behavior.
  \item Species interact moderately, no single component dominates the dynamics.
\end{itemize} & 
\vspace{-0.6cm} \begin{itemize}[]
  \item ECM matrix provides a stable baseline. With a coarse grid and moderately large $\Delta t$, the $\Delta t/h^2$ scaling in other matrices becomes small, reducing their condition numbers and allowing the constant $k(\mathbf{A}_e)$ to dominate.
\end{itemize}\\
\bottomrule
\end{tabular}
\label{tab:condition-numbers}
\end{table}

\begin{table}[htpb]
\centering
\renewcommand{\arraystretch}{1.5}
\caption{Analysis of Condition Numbers in Biological Models (Part 2). Continued from Table~\ref{tab:condition-numbers}.}
\begin{tabular}{@{}c p{3.5cm} p{2cm} p{4cm} p{4cm}@{}}
\toprule
\textbf{Case} & \textbf{Parameter Values} & \textbf{Dominant Term} & \textbf{Biological/Numerical Implication} & \textbf{Condition Number Behavior} \\
\midrule

\textbf{E} & 
\vspace{-0.6cm} \begin{itemize}[]
  \item Large $ \left\| \mathbf{c}^{g,N}\right\|^{2}_2 $
  \item High $p_f$, $w_f$, $p_m$, $w_m$
\end{itemize} & 
$ k(\mathbf{A}_g) $ likely dominates, but all $k(\mathbf{A}_i)$ are reduced & 
\vspace{-0.6cm} \begin{itemize}[]
  \item Strong growth factor signal influences fibroblasts and macrophages significantly.
\end{itemize} & 
\vspace{-0.6cm} \begin{itemize}[]
  \item Lower condition numbers overall due to the strong growth factor influence increasing the damping terms in $k(\mathbf{A}_f)$ and $k(\mathbf{A}_m)$. No single matrix excessively dominates.
\end{itemize} \\

\textbf{F} & 
\vspace{-0.6cm} \begin{itemize}[]
  \item High $ \lambda_i $ (all species)
  \item Large $ \Delta t $
\end{itemize} & 
Possibly $ k(\mathbf{A}_e) $ & 
\vspace{-0.6cm} \begin{itemize}[]
  \item High decay rates reduce active dynamics of all species.
\end{itemize} & 
\vspace{-0.6cm} \begin{itemize}[]
  \item System stabilizes due to high decay.  Large $\Delta t$ and $\lambda_i$ reduce the other condition numbers, potentially allowing the constant $k(\mathbf{A}_e)$ to dominate.
\end{itemize} \\

\bottomrule
\end{tabular}
\label{tab:condition-numbers-2}
\end{table}

\subsection{General discussion}
From \eqref{Eq:SystCondNbr}, we see that the system's condition number bound at time \(N+1\) is given roughly by the maximum over the condition number of the four variables (growth factor, fibroblasts, and macrophages), and a constant corresponding to the condition number of the ECM. The bound is of the general form: $1 + ((D_{u} \cdot \zeta^{u}_{2})/((1/\Delta t +\lambda_{u} )\cdot \zeta^{u}_{1} + 2p_{u} w_{u}\zeta^{u}_{2} \left\|\mathbf{c}^{g, N}\right\|^{2}_{2}))  h^{-2},\; u\in\{g, f,m\}$ except for the growth factor which does not contain the additional proliferation term at the denominator. In the above term, $\zeta^{u}_{1}, \zeta^{u}_{2}$ are the general expressions (\textit{i.e.}, not specific to any variable) of those variables defined in Lemmas~\ref{lemma:mass-mat:stiff-mat} and \ref{Lemma:doubleECMKnownSol}, and Remarks~\ref{rem:Af} and \ref{rem:Am}. We emphasize that the extra term in the denominator of the fibroblasts and macrophages (term which contains the proliferation rates $2p_{_{f}}$, $2p_{_{m}}$, multiplied by the volume fraction coefficients $w_{_{f}},w_{_{m}}$, multiplied by the L$^{2}$-norm of the nodal values of the growth factor at time \(t = N\)) has a damping effect on the spatial discretization. An increase/decrease in the values of these extra terms affects the condition number of the fibroblasts and macrophages. 

\subsubsection{The impact of time-step vs. mesh size on the condition number}
In this section, we discuss how the choice of time-step $\Delta t$, and the mesh size $h$, affect the condition number of the system. The key interplay is between the $\left(\frac{1}{\Delta t} + \lambda\right)$ term in the denominator of the condition number and the $h^{-2}$ factor. The main points are summarized in Table~\ref{tab:gridSize-stepSize_summary}, and discussed in more detail in the three cases below.
\paragraph{Case 1: $\Delta t \ll h^2$.}
When $\Delta t$ is very small, $1/\Delta t$ is very large. This makes the term $\left(\frac{1}{\Delta t} + \lambda\right)$ large and \(1/\Delta t \gg \{\lambda \,\,(+ \,\,\text{proliferation term})\}\), thereby increasing the denominator and making the condition number approximately equal to $\Delta t \cdot D_{(g,f,m)})\cdot h^{-2}$. Consequently, for a fixed $h$, the bound $1 + (\text{constant}/\text{large denominator}) \cdot h^{-2}$ becomes smaller, as the very small value of \(\Delta t\) cushions the effect of $h^{-2}$ making the system better conditioned. However, using an extremely small $\Delta t$ might be computationally expensive, even though it helps to ``tame'' the diffusive stiffness coming from the $h^{-2}$ factor.

\paragraph{Case 2: $\Delta t \gg h^2$.}
If $\Delta t$ is large, then $1/\Delta t$ is small and the contribution of $\lambda$ (and any proliferation term) will matter more. A smaller denominator (caused by a small $1/\Delta t$) makes the quotient larger, amplifying the $h^{-2}$ contribution and potentially leading to a larger condition number. Numerically, this means that if the time-step is too large relative to the spatial resolution (or the diffusion rates), the discrete system can become very ill-conditioned.

\paragraph{Case 3: $\Delta t \approx h^{2}$.}
In many practical problems involving diffusion, one is careful to choose $\Delta t$ and $h$ so that $\Delta t/h^2$ remains bounded (or as close to unity as possible). In such balanced cases, the denominator (dominated by $1/\Delta t$) grows as $\Delta t$ is reduced relative to $h^{2}$, thus keeping the overall contribution modest. However, the proliferative terms for fibroblasts and macrophages (i.e., the terms with $p_{f},p_{m}$) provide extra damping; biologically, a system with fast decay and/or strong proliferation is less ``sensitive'' to the discretization, and numerically, the matrices are easier to invert (\textit{i.e}., better conditioned)~\cite{johnston2014scratch, zhang2006morphogen}.

\begin{table}[htbp]
\centering
\caption{Effect of Time-step \( \Delta t \) Relative to Spatial Grid \( h \) on Condition Number}
\label{tab:gridSize-stepSize_summary}
\begin{tabular}{@{}p{2cm} p{5.5cm} p{5.5cm}@{}}
\toprule
\textbf{Condition} & \textbf{Numerical Implication} & \textbf{Biological Implication} \\
\midrule
 \( \Delta t \ll h^2 \) & 
\vspace{-0.6cm}\begin{itemize}[]
    \item Large denominator \( \left( \frac{1}{\Delta t} + \lambda \right) \), so the condition number is small.
    \item System is well-conditioned.
    \item Computational cost is high due to tiny time-steps.
\end{itemize} & 
\vspace{-0.6cm}\begin{itemize}[]
    \item Dynamics change slowly per time-step.
    \item Captures fast diffusion accurately.
\end{itemize} \\[3ex]
 \( \Delta t \gg h^2 \) 
 & 
\vspace{-0.6cm}\begin{itemize}[]
    \item Small denominator leads to larger condition number.
    \item System becomes ill-conditioned due to amplified \( h^{-2} \) effect.
\end{itemize} & 
\vspace{-0.6cm}\begin{itemize}[]
    \item Temporal dynamics become too coarse, potentially missing fast biological processes.
\end{itemize} \\[2ex]

 \( \Delta t \approx h^2 \)  & 
\vspace{-0.6cm}\begin{itemize}[]
    \item Condition number is moderate and bounded.
    \item Numerically stable and efficient.
\end{itemize} & 
\vspace{-0.6cm}\begin{itemize}[]
    \item Captures diffusion and decay dynamics realistically.
    \item Proliferative damping helps stabilize.
\end{itemize} \\
\bottomrule
\end{tabular}
\label{tab:time-step-grid-condition-number}
\end{table}

\section{Summary}
\label{Sec:summary}
The condition number capture how sensitive the discretized model is to perturbations, and this sensitivity is influenced by both biological and numerical parameters in the mathematical model (and in its numerical discretization).  
While there are many published studies investigating condition numbers for local PDE equations and systems of equations discretized via finite element methods, such studies are almost inexistent for non-local PDE systems (that are becoming more-and-more common in the context of biological and medical applications). In this work, we aimed to address this knowledge gap.

The results of this study, focused on the non-local PDE system \eqref{EqModel} that was initially introduced in~\cite{AdebayoEtAl2023,AdebayoTrucuEftimie2025}, showed that high diffusion tends to increase model sensitivity; this makes intuitive sense since rapidly diffusing substances like cells or molecules lead to tightly coupled spatial points~\cite{schmitt2013diffusion, mcgrath2019tightly}. Furthermore as previously mentioned above, the decay and proliferation terms serve as stabilizing ``mass-matrix'' contributions, as fast decay/proliferation can damp the rapid-varying spatial gradients, thus improving numerical conditioning. Lastly, the term $1/\Delta t$ acts as a sort of ``numerical regularizer''; hence, choosing a small $\Delta t$ (that scales appropriately with $h$) increases the denominator and improves the conditioning of the system. 

In biological terms, the results of this theoretical study suggest that high diffusion rates of chemical species or cells, accompanied by low secretion/proliferation rates and/or low decay rates, may lead to numerically stiff systems. Conversely, if we have very high decay or proliferation rates in the biological system, then even with little diffusion, the system remains numerically manageable provided that the time step is chosen appropriately. 

Note that while we talked about the importance of the proliferation rates ($p_{f,m}$), we included here implicitly also the role played by the terms $w_{f,m}$ that multiply the proliferation rates in the denominator of condition number. These terms are the volume fraction indices, and appear only in the non-local advection fluxes (to avoid cell overcrowding). The approximations that we made for $\mathbf{A}_{f}$ (in Remark~\ref{rem:Af}) and for $\mathbf{A}_{m}$ (in Remark~\ref{rem:Am}) to simplify our calculations, while acceptable numerically (see Figure~\ref{Fig-L2NormDifference}), also caused the disappearance of all other terms in the non-local fluxes (e.g., haptotactic rates $\mu_{f,m}$) from the estimated condition numbers. Hence, future work will have to focus on developing more accurate estimates for the condition numbers in the context of non-symmetric matrices arising from implicit schemes applied to the non-local adhesion fluxes (see the discussion in Remark~\ref{rem:Af} about the term  \(-\mu_{_{f}}\int_{\Omega}\spta{A_{f}^{N}[u],f^{N+1}\,\nabla  v_{_{f}}} \text{d}\mathbf{x} \)).  

The current study used FEniCS to implement the finite element method (FEM) discretization of the non-local PDE model~\eqref{EqModel}. 
{But this FEniCS implementation was only meant to be a proof of concept for this class of non-local PDE models (other software could also be used; see the discussion in~\cite{MarguetEftimieLozinski}). We acknowledge that further numerical investigation is required, including the investigation of potential benefits of implementing adaptive mesh refinements for steep gradient regions. Addressing this matter will require the extension of the current FEM discretisation framework for these non-local models (as FEniCS implementation of this class of non-local models is not always straightforward, as shown in~\cite{MarguetEftimieLozinski} for a very simple non-local equation).}\\

\textbf{\large{Acknowledgements.}} This work (OA,RE) was supported by the ANR grant ANR-21-CE45-0025-01.

\textbf{\large{Conflict of interest.}} The authors declare no conflict of interest.

\begin{appendix}

\section{Appendix}

\subsection{Numerical discretization details}
\label{appendix-full-discretization}
Here, we give the full details of the temporal and spatial discretization of model~\eqref{EqModel}. For the spatial discretization, we make use of the finite element methods as outlined in~\eqref{Eq:space-discetization} and denote nodal values $c_{\tau}^{u}:= c_{\tau}^{u, N+1}$, \(u\in\{g,f,m,e\}\). For simplicity, we also denote \(\psi_{\tau} := \psi_{\tau}(\mathbf{x})\).
For the time discretization, we make use of the backward ({i}mplicit) Euler method as outlined in (\ref{Eq:time-discetization}). The model now reads as follows: \\

{Find} $\{c^{g}_{\tau}\}_{\tau=1}^l, \{c^{f}_{\tau}\}_{\tau=1}^l, \{c^{m}_{\tau}\}_{\tau=1}^l, \{c^{e}_{\tau}\}_{\tau=1}^l$ \emph{such that}
\begin{subequations} 
\label{WEqModel3} 
\begin{align*} 
\sum_{\tau=0}^{l} \frac{c_{\tau}^{g}- c_{\tau}^{g, N}}{\Delta t}\int\limits_{\Omega}\psi_{\tau}\psi_{o}\,d\mathbf{x} =&\int\limits_{\Omega} \sum_{\tau=0}^{l} \left[- D_{_{g}}c_{\tau}^{g}\spta{\nabla \psi_{\tau},\nabla \psi_{o}} + \left(p_{_{g}}^{f} {c_{\tau}^{f, N}}
+ p_{_{g}}^{m} {c_{\tau}^{m, N}}- \lambda_{_{g}} c_{\tau}^{g}
\right)\psi_{\tau}\psi_{o} \right]d\mathbf{x}\\
&\forall\; o\in \{1, \ldots, l\},\\[1ex]
 \sum_{\tau=0}^{l} \frac{c_{\tau}^{f} - c_{\tau}^{f, N}}{\Delta t}\int\limits_{\Omega}\psi_{\tau}\psi_{o}\,d\mathbf{x} = &\int\limits_{\Omega} \sum_{\tau=0}^{l}c_{\tau}^{f} \Bigg[-D_{_{f}} c_{\tau}^{f}\spta{\nabla \psi_{\tau},\nabla \psi_{o}} + \mu_{_{f}}\spta{ A_{_{f}}^{^{N}}, \psi_{\tau} \nabla\psi_{o}}\\
&+\Bigg( p_{_{f}}\left(\sum_{s = 1}^{l} c_{q}^{g, N}\psi_{q}\right) \left(1 - \sum_{q = 1}^{l}\psi_{s} \left[c_{s}^{g,N}+ c_{s}^{f} + c_{s}^{m,N}+ c_{s}^{e,N}\right]\right) \\
&-\lambda_{_{f}}\Bigg)\psi_{\tau} \psi_{o}\Bigg]\,d\mathbf{x}\quad 
\forall\; o\in \{1, \ldots, l\},\\
\sum_{\tau=0}^{l} \frac{c_{\tau}^{m} - c_{\tau}^{m, N}}{\Delta t}\int\limits_{\Omega}\psi_{\tau}\psi_{o}\,d\mathbf{x} = &\int\limits_{\Omega} \sum_{\tau=0}^{l}c_{\tau}^{m} \Bigg[-D_{_{m}} c_{\tau}^{m}\spta{\nabla \psi_{\tau},\nabla \psi_{o}} + \mu_{_{m}}\spta{ A_{_{m}}^{^{N}}, \psi_{\tau} \nabla\psi_{o}}\\
&+\Bigg( p_{_{m}}\Bigg(\sum_{q = 1}^{l} c_{q}^{g, N}\psi_{q}\Bigg) \Bigg(1 - \sum_{s = 1}^{l}\psi_{s} \left[c_{s}^{g,N}+ c_{s}^{f,N} + c_{s}^{m}+ c_{s}^{e,N}\right]\Bigg) \\
&-\lambda_{_{m}}\Bigg)\psi_{\tau} \psi_{o}\Bigg]\,d\mathbf{x}\quad 
\forall\; o\in \{1, \ldots, l\},
\end{align*}
\begin{align*}
\sum_{\tau=0}^{l} \frac{c_{\tau}^{e} - c_{\tau}^{e, N}}{\Delta t}\int\limits_{\Omega}\psi_{\tau}\psi_{o}\,d\mathbf{x}= &\int\limits_{\Omega}\Bigg[\sum_{\tau=0}^{l}c_{\tau}^{e}\Bigg[p_{e}\left(\sum_{q = 1}^{l} c_{q}^{f, N}\psi_{q}\right)\left(1 - \sum_{s = 1}^{l}\psi_{s} \left[c_{s}^{g,N}+ c_{s}^{f,N} + c_{s}^{m,N}+ c_{s}^{e}\right]\right)\\
&- \sum_{r = 1}^{l}\psi_{r}\left[\alpha_{_{f}}c_{r}^{f, N} +\alpha_{_{m}}c_{r}^{m, N}\right] + \alpha_{e}\Bigg]\psi_{\tau} + e_{c}\Bigg]\psi_{o}\,d\mathbf{x}\quad
\forall\; o\in \{1, \ldots, l\}.
\end{align*}

\end{subequations}
\subsection{Computational implementation details}
\label{Appendix:num-det}
Here we provide some additional information related to the software used, the initial conditions, and the parameter values related to Figure~\ref{Fig-L2NormDifference}. \\
For the numerical simulations performed throughout this study, we used FEniCS (an open source finite element computing software: \url{https://fenicsproject.org/}). \\
The time interval considered in this study was $I = [0,100]$. At the boundary, we assumed zero-flux boundary conditions (see also equations \eqref{EqBC:ExisUniq}). In fact, the boundary conditions for fibroblasts and macrophages are now reduced to simple homogenous Neumann boundary conditions after approximating non-local adhesion terms \(A[u]_{f}^{N}\) and \(A[u]_{m}^{N}\) as shown in \eqref{Eq:nonlocal:Approx-f} and \eqref{Eq:nonlocal:Approx-m}), respectively. \\
For the numerical simulations we used the following initial conditions:
\begin{eqnarray}
 \label{InitCond:wound:a}
    \!\!\!\!\!g(\mathbf{x},0) &=& 0.1,\nonumber\\
    \label{InitCond:wound:b}
   \!\!\!\!\!f(\mathbf{x},0) &=& 0.2\!\left[\big(0.5+0.5\tanh(20x_{_1}\!\!-3)\big)\!+\!\big(0.5+0.5 \tanh(-20x_{_1}\!\!-3)\big)\right] \!,\nonumber\\
   \label{InitCond:wound:c}
   \!\!\!\!\! m(\mathbf{x},0) &=& 0.5\!\left[\big(0.5+0.5 \tanh(20x_{_1}\!\!-3)\big)\!+\!\big(0.5+0.5 \tanh(-20x_{_1}\!\!-3)\big)\right]\!,\\
   \!\!\!\!\!e(\mathbf{x},0) &=& 1.0\!\left[\big(0.5+0.5 \tanh(20x_{_1}\!\!-3)\big)\!+\!\big(0.5+0.5\tanh(-20x_{_1}\!\!-3)\big)\right]\!,\nonumber
    \label{InitCond:wound:d}
\end{eqnarray}
on the square domain $\Omega = [-1,1]^{^{2}}$ with \(\Delta t = 0.2\). \\
Finally, the parameter values used for the simulations are listed in Table~\ref{Table-Param}. 
\begin{longtable}{p{1.0cm}p{1.0cm}p{7cm}p{1.5cm}}

\hline
Param. & Value & Description & Reference\\
\hline
$D_{_{g}}$ & $0.0035 $ & Diffusion coeff. for growth-factor population & \cite{DomschkeTrucuGerischChaplain}\\
$D_{_{f}}$ & $0.0008$ & Diffusion coeff. for fibroblast population & \cite{DomschkeTrucuGerischChaplain, AdebayoTrucuEftimie2025}\\
$D_{_{m}}$ & $0.0008 $ & Diffusion coeff. for macrophage population & \cite{DomschkeTrucuGerischChaplain, AdebayoTrucuEftimie2025}\\
$\lambda_{_{g}}$ & $0.2$ & Decay rate of growth-factor population & \cite{AdebayoEtAl2023, AdebayoTrucuEftimie2025}\\
$\lambda_{_{f}}$ & $0.025 $ & Apoptotic rate of fibroblast population & \cite{AdebayoEtAl2023, AdebayoTrucuEftimie2025} \\
$\lambda_{_{m}}$ & $0.025$ &  Apoptotic rate of macrophages population & \cite{AdebayoEtAl2023, AdebayoTrucuEftimie2025}\\
$p_{_{g}}^f$ & $0.2$ & Secretion rate of growth-factor by fibroblasts & \cite{AdebayoEtAl2023, AdebayoTrucuEftimie2025}\\
$p_{_{g}}^m$ & $0.2$  & Secretion rate of growth-factor by macrophages & \cite{AdebayoEtAl2023, AdebayoTrucuEftimie2025}\\
$p_{_f}$ & $5.0$  & Proliferation rate of fibroblasts population depending on the growth factor & \cite{AdebayoEtAl2023, AdebayoTrucuEftimie2025}\\
$p_{_m}$ & $5.0$  &  Proliferation rate of macrophages population depending on the density of the growth factor & \cite{AdebayoEtAl2023}\\
$\alpha_{_{f}}$ & $0.015$ & Degradation rate of ECM by fibroblasts & \cite{DomschkeTrucuGerischChaplain, AdebayoTrucuEftimie2025}\\
$\alpha_{_{m}}$ & $0.015$ & Degradation rate of ECM by macrophages & \cite{DomschkeTrucuGerischChaplain, AdebayoTrucuEftimie2025}\\
{$\alpha_{_{e}}$ }& {$0.05$} & {Degradation rate of ECM by other cells} & \cite{AdebayoEtAl2023, AdebayoTrucuEftimie2025}\\
{$e_{_{c}}$} & {$0.1$} &  {ECM secreted by other cells} & \cite{AdebayoEtAl2023, AdebayoTrucuEftimie2025}\\
$p_{_{e}}$ & $5.0$ & Remodelling rate of ECM population & \cite{AdebayoEtAl2023, AdebayoTrucuEftimie2025}\\
$w_{_{g}}$ & $1$ & Fraction of physical space occupied by growth factor &  \cite{AdebayoEtAl2023, Robinetal2019, AdebayoTrucuEftimie2025}\\
$w_{_{f}}$ & $1$ & Fraction of physical space occupied by fibroblasts & \cite{AdebayoEtAl2023, Robinetal2019, AdebayoTrucuEftimie2025}\\
$w_{_{m}}$ & $1$ & Fraction of physical space occupied by macrophages & \cite{AdebayoEtAl2023, Robinetal2019, AdebayoTrucuEftimie2025}\\
$w_{_{e}}$ & $1$ & Fraction of physical space occupied by ECM & \cite{AdebayoEtAl2023, Robinetal2019, AdebayoTrucuEftimie2025}\\
${S}_{_{ff}}^\max$ & $0.2$ & Maximum strength of fibroblast-fibroblast adhesive junction & \cite{AdebayoEtAl2023, AdebayoTrucuEftimie2025}\\
${S}_{_{fm}}^\max$ & $0.1$ & Maximum strength of fibroblast-macrophages adhesive junction & \cite{GerischChaplain, AdebayoTrucuEftimie2025}\\
${S}_{_{mf}}^\max$ & $0.1$ & Maximum strength of macrophages-fibroblast adhesive junction & \cite{GerischChaplain, AdebayoTrucuEftimie2025}\\
${S}_{_{mm}}^\max$ & $0.2$ & Maximum strength of macrophages-macrophages adhesive junction & \cite{AdebayoEtAl2023, AdebayoTrucuEftimie2025}\\
${S}_{_{fe}}^\max$ & $0.1$ & Maximum strength of fibroblast-ECM adhesive junction & \cite{GerischChaplain, AdebayoTrucuEftimie2025}\\
${S}_{_{me}}^\max$ & $1.0$ & Maximum strength of macrophages-ECM adhesive junction & \cite{GerischChaplain, AdebayoTrucuEftimie2025}\\
$\mu_{_f}$ & $0.08$ & Haptotactic rate of the fibroblasts & \cite{AdebayoEtAl2023, AdebayoTrucuEftimie2025} \\
$\mu_{_m}$ & $0.08$ & Haptotactic rate of the macrophages &\cite{AdebayoEtAl2023, AdebayoTrucuEftimie2025} \\
$R$ & $0.1$ & Sensing radius for the non-local interaction & \cite{GerischChaplain}\\
$\sigma$ &$0.04$& standard deviation of the kernels&\cite{AdebayoEtAl2023, AdebayoTrucuEftimie2025}\\
\hline
\caption{ Summary of dimensionless model parameters and the baseline values used for the numerical simulations.}
\label{Table-Param}
\end{longtable}

\subsection{Spatio-temporal simulations corresponding to Figure 1}\label{Sect:Appendix3}
In Figure~\ref{fig:model_dynamic}, we show the spatio-temporal dynamics of the non-local model ~\eqref{EqModel}, for the same parameter values as in Figure~\ref{Fig-L2NormDifference}. We see that an initial linear cut in the tissue domain (described by a sharp decrease at $t=0$ in the middle of the domain, in the densities of ECM, as well as fibroblasts and macrophages) heals quickly due to an increased number of macrophages and fibroblasts (at $t=3$). By time $t=10$ the cut vanishes (i.e., ECM approached a spatially homogeneous steady state).
More numerical simulations (for different parameter values) showing the varied spatio-temporal dynamics of this non-local model~\eqref{EqModel} can be found in references~\cite{AdebayoEtAl2023, AdebayoTrucuEftimie2025}. 
\begin{figure}
    \centering
    \includegraphics[width=1.0\linewidth]{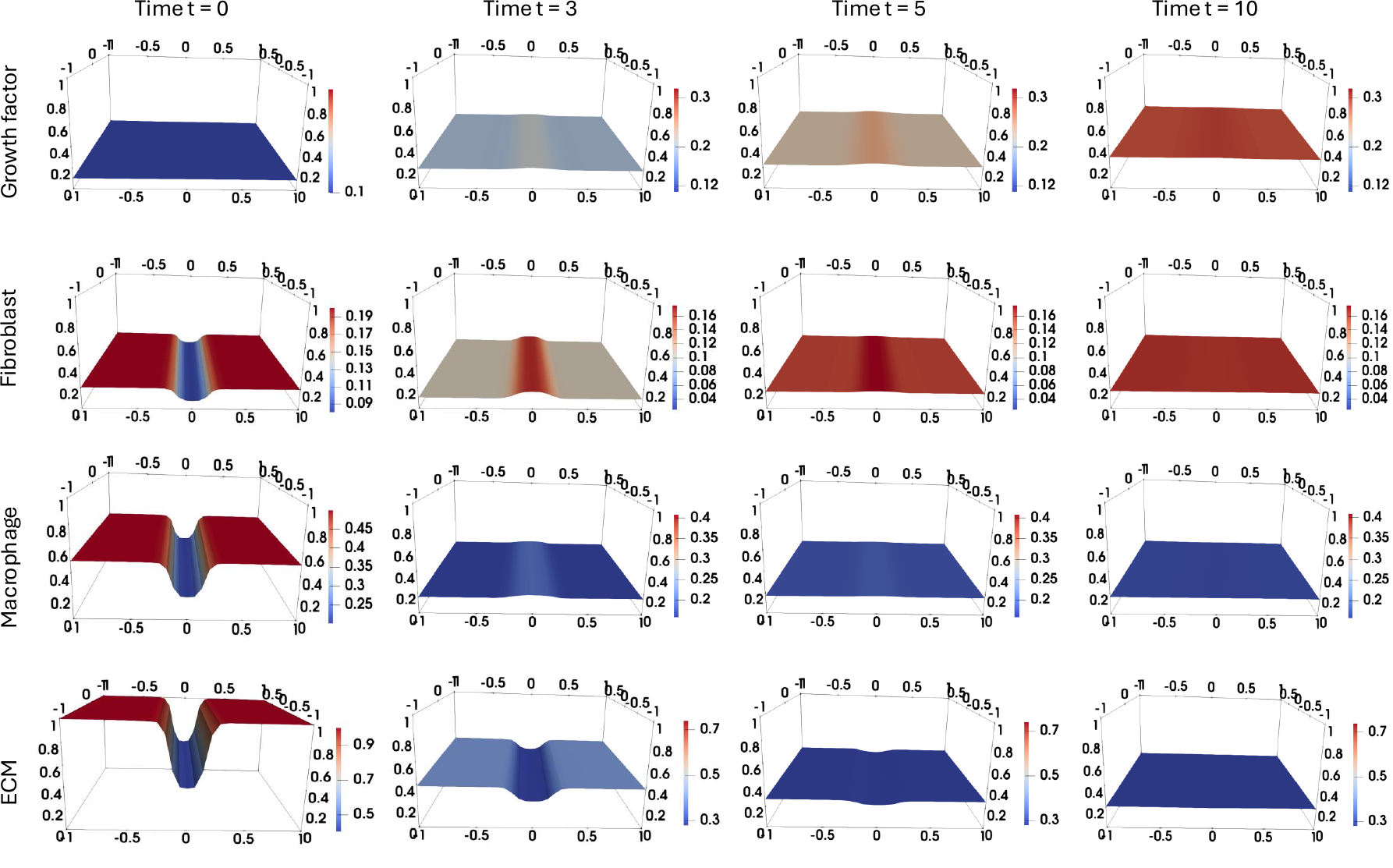}
    \caption{Numerical simulation for the model corresponding to Figure~\ref{Fig-L2NormDifference}. The initial conditions, parameter values and other numerical details are shown in Appendices~\ref{appendix-full-discretization} and \ref{Appendix:num-det}.}
    \label{fig:model_dynamic}
\end{figure}
\end{appendix}

\bibliographystyle{plain}
\bibliography{cond}

\begin{thebibliography}{10}

\bibitem{AdebayoEtAl2023}
O.E. Adebayo, S.~Urcun, G.~Rolin, S.~Bordas~D. Trucu, and R.~Eftimie.
\newblock Mathematical investigation of normal and abnormal wound healing
  dynamics: local and non-local models.
\newblock {\em Bull. Math. Biol.}, 20:17446–17498, 2023.

\bibitem{AdebayoTrucuEftimie2025}
Olusegun~E. Adebayo, Dumitru Trucu, and Raluca Eftimie.
\newblock Analytical investigation of a non-local mathematical model for normal
  and abnormal wound healing.
\newblock {\em Discrete and Continuous Dynamical Systems - B},
  30(7):2401--2428, 2025.

\bibitem{Arbogast1995}
Todd Arbogast and Mary~F. Wheeler.
\newblock A characteristics-mixed finite element method for advection-dominated
  transport problems.
\newblock {\em SIAM Journal on Numerical Analysis}, 32(2):404--424, 1995.

\bibitem{Astuto2023}
Clarissa Astuto, D.~Boffi, J.~Haskovec, P.~Markowich, and G.~Russo.
\newblock Asymmetry and condition number of an elliptic-parabolic system for
  biological network formation.
\newblock {\em ArXiv}, abs/2301.12926, 2023.

\bibitem{Bathe2014}
K.J. Bathe.
\newblock {\em Finite Element Procedures}.
\newblock Prentice Hall, 2014.

\bibitem{BitsouniChaplainEftimie2017}
V.~Bitsouni, M.A.J. Chaplain, and R.~Eftimie.
\newblock Mathematical modeling of cancer invasion: the multiple roles of tgf-b
  pathway on tumour proliferation and cell adhesion.
\newblock {\em Mathematical Models and Methods in Applied Sciences},
  27(3):1929--1962, 2017.

\bibitem{BitsouniEtAl2018}
V.~Bitsouni, D.~Trucu, M.~A.~J. Chaplain, and R.~Eftimie.
\newblock Aggregation and travelling wave dynamics in a two-population model of
  cancer cell growth and invasion.
\newblock {\em Mathematical Medicine and Biology: A Journal of the IMA},
  35:541--577, 2018.

\bibitem{dalhquist74}
A.~Bjorck and G.~Dahlquist.
\newblock {\em Numerical Methods}.
\newblock Prentice-Hall, 1974.

\bibitem{Chavent2003}
G.~Chavent, A.~Younes, and P.~Ackerer.
\newblock On the finite volume reformulation of the mixed finite element method
  for elliptic and parabolic pde on triangles.
\newblock {\em Computer Methods in Applied Mechanics and Engineering},
  192:655--682, 2003.

\bibitem{DomschkeTrucuGerischChaplain}
P.~Domschke, D.~Trucu, A.~Gerisch, and M.~Chaplain.
\newblock Mathematical modelling of cancer invasion: implications of cell
  adhesion variability for tumour infiltrative growth patterns.
\newblock {\em Journal of Theoretical Biology}, 361:41--60, 2014.

\bibitem{EckardtPainterSurulescuZhigun}
M.~Eckardt, K.J. Painter, C.~Surulescu, and A.~Zhigun.
\newblock Nonlocal and local models for taxis in cell migration: a rigorous
  limit procedure.
\newblock {\em J. Math. Biol.}, 81:1251--1298, 2020.

\bibitem{EISENTRAGER20202289}
S.~Eisenträger, E.~Atroshchenko, and R.~Makvandi.
\newblock On the condition number of high order finite element methods:
  Influence of p-refinement and mesh distortion.
\newblock {\em Computers and Mathematics with Applications}, 80(11):2289--2339,
  2020.

\bibitem{EllefsenRodriguez}
E.~Ellefsen and N.~Rodriguez.
\newblock Nonlocal mechanistic models in ecology: Numerical methods and
  parameter inferences.
\newblock {\em Appl. Sci.}, 13(19):10598, 2023.

\bibitem{ErnGuermond2006}
Alexandre Ern and Jean-Luc Guermond.
\newblock Evaluation of the condition number in linear systems arising in
  finite element approximation.
\newblock {\em ESAIM: Mathematical Modelling and Numerical Analysis},
  40:29–48, 2006.

\bibitem{FengLewis2018}
Xiaobing~H. Feng and T.~Lewis.
\newblock Nonstandard local discontinuous galerkin methods for fully nonlinear
  second order elliptic and parabolic equations in high dimensions.
\newblock {\em Journal of Scientific Computing}, 77:1534--1565, 2018.

\bibitem{GerischChaplain}
A.~Gerisch and M.~Chaplain.
\newblock Mathematical modelling of cancer cell invasion of tissue: local and
  non-local models and the effect of adhesion.
\newblock {\em Journal of Theoretical Biology}, 250:684--704, 2008.

\bibitem{GlimmConditionNbrNonlocal}
T.~Glimm and J.~Zhang.
\newblock Numerical approach to a nonlocal advection-reaction-diffusion model
  of cartilage pattern formation.
\newblock {\em Mathematical and Computational Applications}, 25:36, 2020.

\bibitem{Hong2024}
Youngjoon Hong, Seungchan Ko, and Jae~Yong Lee.
\newblock Error analysis for finite element operator learning methods for
  solving parametric second-order elliptic pdes.
\newblock {\em arXiv}, abs/2404.17868, 2024.

\bibitem{Johnson1987}
C.~Johnson.
\newblock {\em Numerical Solution of Partial Differential Equations by the
  Finite Element Method}.
\newblock Cambridge University Press, New York, 1987.

\bibitem{johnston2014scratch}
Stuart~T. Johnston, Matthew~J. Simpson, and D.~L.~Sean McElwain.
\newblock How much information can be obtained from tracking the position of
  the leading edge in a scratch assay?
\newblock {\em BMC Systems Biology}, 8(1):1--11, 2014.

\bibitem{Kannan2014}
R.~Kannan, S.~Hendry, N.J. Higham, and F.~Tisseur.
\newblock Detecting the causes of ill-conditioning in structural finite element
  models.
\newblock {\em Computers and Structures}, 133:79--89, 2014.

\bibitem{KimPark1998}
Chang-Geun Kim and Jungho Park.
\newblock The condition number of stiffness matrix under p-version of the fem.
\newblock {\em Kangweon-Kyungki Mathematical Journal}, 6(1):17--26, 1998.

\bibitem{LeeMogilner2001}
C.T. Lee, M.F. Hoopes, J.~Diehl, W.~Gilliland, G.~Huxel, E.V. Leaver,
  K.~McCann, J.~Umbanhowar, and A.~Mogilner.
\newblock Non-local concepts and models in biology.
\newblock {\em J. Theor. Biol.}, 210(2):201--219, 2001.

\bibitem{Mabuza2018Local}
Sibusiso Mabuza, J.~Shadid, and D.~Kuzmin.
\newblock Local bounds preserving stabilization for continuous galerkin
  discretization of hyperbolic systems.
\newblock {\em J. Comput. Phys.}, 361:82--110, 2018.

\bibitem{MarguetEftimieLozinski}
J.~Marguet, R.~Eftimie, and A.~Lozinski.
\newblock Numerical approaches for non-local transport-dominated pde models
  with applications to biology.
\newblock {\em Computational and Applied Mathematics}, 44:199, 2025.

\bibitem{mcgrath2019tightly}
J.~McGrath and A.~R. Smith.
\newblock Spatial pattern formation in reaction–diffusion models: a
  computational perspective.
\newblock {\em Mathematics of Computation}, 88(317):1591--1617, 2019.

\bibitem{PainterHillenPotts2024}
K.J. Painter, T.~Hillen, and J.R. Potts.
\newblock Biological modeling with nonlocal advection–diffusion equations.
\newblock {\em Mathematical Models and Methods in Applied Sciences},
  34(1):57--107, 2024.

\bibitem{PalMelnik2025}
S.~Pal and R.~Melnik.
\newblock Nonlocal models in biology and life sciences: Sources, developments,
  and applications.
\newblock {\em Physic of Life Reviews}, 52:29--43, 2025.

\bibitem{Parra2021}
E.~R. Parra.
\newblock Methods to determine and analyze the cellular spatial distribution
  extracted from multiplex immunofluorescence data to understand the tumor
  microenvironment.
\newblock {\em Frontiers in Molecular Biosciences}, 8:668340, 2021.

\bibitem{Robinetal2019}
R.~Shuttleworth and D.~Trucu.
\newblock Multiscale modelling of fibres dynamics and cell adhesion within
  moving boundary cancer invasion.
\newblock {\em Bulletin of Mathematical Biology}, 81:2176--2219, 2019.

\bibitem{schmitt2013diffusion}
Juncheng Wei and Matthias Winter.
\newblock {\em Mathematical Aspects of Pattern Formation in Biological
  Systems}.
\newblock Springer, Berlin, 2013.

\bibitem{XueEtAl2015}
M.~Xue and C.~J. Jackson.
\newblock Extracellular matrix reorganization during wound healing and its
  impact on abnormal scarring.
\newblock {\em Advances in Wound Care}, 4:119--136, 2015.

\bibitem{zhang2006morphogen}
Yi~Zhang, Mark Alber, and Stuart~A. Newman.
\newblock Mathematical modeling of vertebrate limb development.
\newblock {\em Bulletin of Mathematical Biology}, 68(5):1169--1192, 2006.

\bibitem{ZhaoEtAl2022}
X.~Zhao, J.~Chen, H.~Sun, Y.~Zhang, and D.~Zou.
\newblock New insights into fibrosis from the ecm degradation perspective: the
  macrophage-mmp-ecm interaction.
\newblock {\em Cell and Bioscience}, 12:117, 2022.

\end{thebibliography}
\end{document}